\documentclass{daj}
\usepackage[T1]{fontenc}
\usepackage[latin9]{inputenc}
\usepackage{amsthm,amsmath}
\usepackage{amssymb,verbatim}

\makeatletter
\numberwithin{equation}{section}
\numberwithin{figure}{section}

\theoremstyle{plain}
\newtheorem{thm}{\protect\theoremname}[section]
  \theoremstyle{definition}
  \newtheorem{example}[thm]{\protect\examplename}
  \theoremstyle{plain}
  \newtheorem{lem}[thm]{\protect\lemmaname}
  \theoremstyle{definition}
  \newtheorem{defn}[thm]{\protect\definitionname}
  \theoremstyle{definition}
  \newtheorem{remark}[thm]{\protect\remarkname}
  \theoremstyle{plain}
  \newtheorem{prop}[thm]{\protect\propositionname}
   \theoremstyle{plain}
   \newtheorem{conj}[thm]{\protect\conjecturename}
   \theoremstyle{plain}
  \newtheorem{cor}[thm]{\protect\corollaryname}
  \theoremstyle{plain}
  \newtheorem*{thm*}{\protect\theoremname}
  \theoremstyle{plain}
  \newtheorem{claim}[thm]{\protect\claimname}
   \theoremstyle{plain}
  \newtheorem*{prop*}{\protect\propositionname}

\makeatother

  \providecommand{\claimname}{Claim}
  \providecommand{\remarkname}{Remark}
  \providecommand{\definitionname}{Definition}
  \providecommand{\examplename}{Example}
  \providecommand{\lemmaname}{Lemma}
  \providecommand{\propositionname}{Proposition}
  \providecommand{\corollaryname}{Corollary}
  \providecommand{\theoremname}{Theorem}
\providecommand{\theoremname}{Theorem}
\providecommand{\conjecturename}{Conjecture}

\dajAUTHORdetails{%
  title = {On the Structure of Subsets of the Discrete Cube with Small Edge Boundary}, %% please capitalize all significant words
  author = {David Ellis, Nathan Keller and Noam Lifshitz},
  plaintextauthor = {David Ellis, Nathan Keller and Noam Lifshitz},
  runningtitle = {Subsets of the Discrete Cube with Small Edge Boundary}, 
  keywords = {Discrete isoperimetric inequalities, stability, discrete cube},
}   %%% END \dajAUTHORdetails

\dajEDITORdetails{%
   year={2018},
   number={9},
   received={12 January 2017},   % received date: example: 7 January 2017
   published={28 May 2018},  % published date
   doi={10.19086/da.3668},       % XXX = number of paper, e.g. da006 for paper#6
}   %%% END \dajEDITORdetails

\begin{document}
\global\long\def\Sym{\mathrm{Sym}}
\global\long\def\Id{\mathrm{Id}}
\global\long\def\Per{\mathrm{Per}}
\global\long\def\f{\mathcal{F}}
\global\long\def\l{\mathcal{L}}
\global\long\def\pn{\mathcal{P}\left(\left[n\right]\right)}
\global\long\def\g{\mathcal{G}}
\global\long\def\s{\mathcal{S}}
\global\long\def\m{\mathcal{M}}
\global\long\def\mh{\mu_{\frac{1}{2}}}
\global\long\def\mp{\mu_{p}}
\global\long\def\j{\mathcal{J}}
\global\long\def\d{\mathcal{D}}
\global\long\def\Inf{\mathrm{Inf}}
\global\long\def\p{\mathcal{P}}
\global\long\def\mpo{\mu_{p_{0}}}
\global\long\def\fpp{f_{p}^{p_{0}}}
\global\long\def\llll{\l_{\mu}}
\global\long\def\h{\mathcal{H}}
\global\long\def\n{\mathbb{N}}
\global\long\def\a{\mathcal{A}}
\global\long\def\b{\mathcal{B}}
\global\long\def\c{\mathcal{C}}
\global\long\def\sf{f_{2}}
\global\long\def\bin{\mathrm{Bin}}
\global\long\def\C{C_{2}}

\begin{frontmatter}[classification=text]
%% EDITOR: this will force the keywords to appear right after the Abstract.
%%   If the abstract is too long and would force the keywords off the
%%   front page, please comment out % [classification=text] above
%%   This way the keywords will be floated on the bottom of the first page
%%   even though the Abstract spills over to the next page.

%%% AUTHOR: Title goes here.  This line is optional.  You must use it
%%   if title has footnote attached or requires nontrivial typesetting,
%%   e.g., inclusion of linebreaks to force nice layout.
%\title{Short Proof of R\"odl's $n^{\log\log n}$ Bound\footnote{This is a footnote to the title}} %% please capitalize all significant words

%%% AUTHOR:
%%% List all authors. If you wish, place grant acknowledgements in \thanks.
%%% In brackets include a short tag for each author.
\author[david]{David Ellis}
\author[nathan]{Nathan Keller\thanks{Supported by the Israel Science Foundation (grant no.
402/13), the Binational US-Israel Science Foundation (grant no. 2014290), and by the Alon Fellowship.}}
\author[noam]{Noam Lifshitz}

%%% AUTHOR: Abstract goes here
\begin{abstract}
The edge isoperimetric inequality in the discrete cube specifies, for each pair of integers
$m$ and $n$, the minimum size $g_n(m)$ of the edge boundary of an $m$-element subset of $\left\{ 0,1\right\}^{n}$; the extremal families (up to automorphisms of the discrete cube) are initial segments of the lexicographic ordering on $\{0,1\}^n$. We show that for any $m$-element subset $\f \subset \{0,1\}^n$ and any integer $l$, if the edge boundary of $\f$ has size at most $g_n(m)+l$, then there exists an extremal family $\g \subset \{0,1\}^n$ such that $|\f \Delta \g| \leq Cl$, where $C$ is an absolute constant. This is best possible, up to the value of $C$. Our result can be seen as a `stability' version of the edge isoperimetric inequality in the discrete cube, and as a discrete analogue of the seminal stability result of Fusco, Maggi and Pratelli \cite{fmp08} for the isoperimetric inequality in Euclidean space.
\end{abstract}
\end{frontmatter}

%%% AUTHOR: body of paper starts here

\section{Introduction}
Isoperimetric inequalities are of ancient interest in mathematics. In general, an isoperimetric inequality gives a lower bound on the `boundary size'
of a set of a given `size', where the exact meaning of these words
varies according to the problem. One of the best-known examples is the isoperimetric inequality for Euclidean space, which states (informally) that among all subsets of $\mathbb{R}^n$ of given volume (whose surface area is defined), a Euclidean ball has the smallest surface area. An exact formulation (actually, one of the versions) of this inequality is as follows.
\begin{thm}\label{thm:Iso-Euclidean}
If $A \subset \mathbb{R}^n$ is a Borel set with Lebesgue measure $\mu(A) < \infty$, then
$$\Per(A) \geq \Per(B),$$
where $B$ is a Euclidean ball in $\mathbb{R}^n$ with $\mu(B)=\mu(A)$.
\end{thm}
Here, $\Per(S)$ denotes the {\em distributional perimeter} of a set $S \subset \mathbb{R}^n$, which is equal to the $(n-1)$-dimensional Lebesgue measure of the topological boundary of $S$, for sufficiently `nice' sets $S$. (E.g., it suffices for $S$ to be a Borel set with finite Lebesgue measure and piecewise smooth topological boundary.)
\medskip

When an isoperimetric inequality is sharp, and the extremal sets are known, it is natural to ask whether the inequality is also `stable' --- i.e., if a set has boundary of size `close' to the minimum, must that set be `close in structure' to an extremal set?

In their seminal work \cite{fmp08}, Fusco, Maggi and Pratelli obtained a stability result for Theorem~\ref{thm:Iso-Euclidean}, confirming a conjecture of Hall~\cite{Hall92}.
\begin{thm}[Fusco, Maggi, Pratelli, 2008]
\label{thm:fmp}
Let $\epsilon >0$. Suppose $A \subset \mathbb{R}^n$ is a Borel set with Lebesgue measure $\mu(A) < \infty$, and with
$$\Per(A) \leq (1+\epsilon)\Per(B),$$
where $B$ is a Euclidean ball with $\mu(B)=\mu(A)$. Then there exists $x \in \mathbb{R}^n$ such that
$$\mu(A \Delta (B + x)) \leq C_n\, \sqrt{\epsilon}\, \mu(A),$$
where $C_n >0$ is a constant depending upon $n$ alone.
\end{thm}

As observed in \cite{fmp08}, Theorem \ref{thm:fmp} is sharp up to the value of the constant $C_n$, as can be seen by taking $S$ to be an ellipsoid with $n-1$ semi-axes of length 1 and one semi-axis of length slightly larger than 1.

\bigskip

\noindent \textbf{Discrete isoperimetric inequalities}

\medskip

In the last fifty years, there has been a great deal of interest in {\em discrete} isoperimetric inequalities. These deal with the boundaries of sets of vertices in graphs. There are two natural measures of the boundary of a set of vertices in a graph: the {\em edge boundary} and the {\em vertex boundary}. If $G = (V,E)$ is a graph, and $A \subset V$ is a set of vertices of $G$, the {\em edge boundary} of $A$ consists of the set of edges of $G$ which join a vertex in $A$ to a vertex in $V \setminus A$; it is denoted by $\partial_{G}(A)$, or by $\partial A$ when the graph $G$ is understood. The {\em vertex boundary} of $A$ consists of the set of vertices of $V \setminus A$ which are adjacent to a vertex in $A$; it is sometimes denoted by $b_{G}(A)$, or by $b(A)$ when the graph $G$ is understood. If $G = (V,E)$ is a graph, the {\em edge isoperimetric problem for $G$} asks for a determination of $\Phi_G(m):=\min\{|\partial A|:\ A \subset V,\ |A|=m\}$, for each integer $m$; similarly, the {\em vertex isoperimetric problem for $G$} asks for a determination of $\Psi_G(m):=\min\{|b(A)|:\ A \subset V,\ |A|=m\}$, for each integer $m$.

An important example of a discrete isoperimetric problem, and the focus of this paper, is the edge-isoperimetric problem for the $n$-dimensional discrete cube, $Q_n$. (We define $Q_n$ to be the graph with vertex-set $\{0,1\}^n$, where two 0-1 vectors are adjacent if they differ in exactly one coordinate.) This isoperimetric problem has numerous applications, both to other problems in mathematics, and in other areas such as communication complexity (see e.g. \cite{Harper}), network science (see \cite{bezrukov}) and game theory (see \cite{Hart}). Hereafter, if $A \subset \{0,1\}^n$, we write $\partial A$ for the edge boundary of $A$ with respect to $Q_n$.

The edge isoperimetric problem for $Q_n$ was solved by Harper \cite{Harper}, Lindsey \cite{Lindsey}, Bernstein \cite{Bernstein}, and Hart \cite{Hart}. Let us describe the solution. We may identify $\{0,1\}^n$ with the power-set $\pn$ of $[n]: = \{1,2,\ldots,n\}$, by identifying a 0-1 vector $(x_1,\ldots,x_n)$ with the set $\{i \in [n]:\ x_i=1\}$. We can then view $Q_n$ as the graph with vertex set $\pn$, where two sets $S, T \subset [n]$ are adjacent if $|S \Delta T|=1$. The {\em lexicographic ordering} on $\pn$ is defined by $S > T$ iff $\min(S \Delta T) \in S$. If $m \in [2^n]$, the {\em initial segment of the lexicographic ordering on $\pn$ of size $m$} is simply the subset of $\pn$ consisting of the $m$ largest elements of $\pn$ with respect to the lexicographic ordering. If $\l \subset \pn$ is an initial segment of the lexicographic ordering, we say $\l$ is {\em lexicographically ordered}. Note that if $m = 2^d$ for some $d \in \mathbb{N}$, then the initial segment of the lexicographic ordering on $\pn$ of size $m$ is the $d$-dimensional subcube $\{S \subset [n]:\ [n-d] \subset S\}$.

Harper, Bernstein, Lindsey and Hart proved the following.
\begin{thm}[The edge isoperimetric inequality for $Q_n$]
\label{thm:edge-iso}
If $\f \subset \p([n])$, then $|\partial \f| \geq |\partial \l|$, where $\l \subset \pn$ is the initial segment of the lexicographic ordering of size $|\f|$.
\end{thm}

Let us describe the extremal families for Theorem \ref{thm:edge-iso}. If $\mathcal{F},\mathcal{G} \subset \p([n])$, we say that $\f$ and $\g$ are {\em weakly isomorphic} if there exists an automorphism $\phi$ of $Q_n$ such that $\g = \phi(\f)$; in this case, we write $\f \cong \g$. Equivalently, $\f,\g \subset \{0,1\}^n$ are weakly isomorphic iff $\g$ can be obtained from $\f$ by permuting the coordinates $1,2,\ldots,n$ and interchanging $0$'s with $1$'s on some subset of the coordinates.
%To be completely formal and explicit, if $\pi \in \textrm{Sym}([n])$ and $S \subset \left[n\right]$, we write $\pi\left(S\right):=\left\{ \pi\left(i\right)\,:\, i\in S\right\}$,
%and if $\f \subset \pn$, we write $\pi\left(\f\right):=\left\{ \pi\left(S\right)\,:\, S\in\f\right\}$. Families $\f,\g \subset \pn$ are said to be \emph{isomorphic} if there exists $\pi \in \Sym([n])$ such that $\g=\pi\left(\f\right)$. If $D \subset [n]$ and $\f \subset \p([n])$, we define $X_D(\f) = \{S \Delta D:\ S \in \f\}$. Families $\f,\g \subset \pn$ are weakly isomorphic iff there exist $\pi \in \Sym([n])$ and $D \subset [n]$ such that $\g=X_D(\pi\left(\f\right))$.
Clearly, weak isomorphism preserves the size of the edge boundary. It is well-known (and easy to check by analyzing known proofs of Theorem \ref{thm:edge-iso}) that equality holds in Theorem \ref{thm:edge-iso} if and only if $\f$ is weakly isomorphic to $\l$. In particular, if $|\f|$ is a power of 2, then equality holds in Theorem \ref{thm:edge-iso} if and only if $\f$ is a subcube.

%We note that Theorem \ref{thm:edge-iso} is sometimes stated in terms of initial segments of the {\em binary ordering} on $\pn$, defined by $S > T$ if $\max(S \Delta T) \in T$; an initial segment of the binary ordering is weakly isomorphic to an initial segment of the lexicographic ordering of the same size. It will be slightly more convenient for us to work with the lexicographic ordering, than with the binary ordering.
%\newline

\medskip

To date, several stability versions of Theorem~\ref{thm:edge-iso} have been obtained. Using a Fourier-analytic argument, Friedgut, Kalai and Naor~\cite{FKN} obtained a stability result for sets of size $2^{n-1}$, showing that if $\f \subset \pn$ with $|\f| = 2^{n-1}$ and $|\partial \f| \leq (1+\epsilon)2^{n-1}$, then $|\f \Delta \mathcal{C}|/2^n= O(\epsilon)$ for some codimension-1 subcube $\mathcal{C}$. (The dependence upon $\epsilon$ here is almost sharp, viz., sharp up to a factor of $\Theta(\log(1/\epsilon))$. Bollob\'as, Leader and Riordan (unpublished) proved an analogous result for $|\f| \in \{2^{n-2},2^{n-3}\}$, also using a Fourier-analytic argument. Samorodnitsky \cite{Samorodnitsky09} used a result of Keevash~\cite{Keevash08} on the structure of $r$-uniform hypergraphs with small shadows, to prove a stability result for all $\f \subset \pn$ with $\log_2|\f| \in \mathbb{N}$, under the rather strong condition $|\partial \f| \leq (1+O(1/n^4))|\partial \l|$. In \cite{Ellis}, the first author proved the following stability result (which implies the above results), using a recursive approach and an inequality of Talagrand \cite{Talagrand} (which is proved via Fourier analysis).
\begin{comment} The proof in \cite{Ellis} formally is phrased in terms of a recursive algorithm, not induction on n, though it could be rephrased in terms of an induction on n. It also doesn't use any Fourier analysis directly, but appeals to the result of Talagrand as a black box (Talagrand's result is proved using Fourier analysis.) 
\end{comment}

\begin{thm}[\cite{Ellis}]
\label{thm:e}
There exists an absolute constant $c>0$ such that the following holds. Let $0 \leq \delta < c$. If $\f \subset \pn$ with $|\f| = 2^{d}$ for some \(d \in \mathbb{N}\), and $|\f \Delta \mathcal{C}| \geq \delta 2^d$ for all $d$-dimensional subcubes $\mathcal{C} \subset \pn$, then
$$|\partial \f| \geq |\partial \mathcal{C}| +2^d \delta \log_{2}(1/\delta).$$
\end{thm}
As observed in \cite{Ellis}, this result is best possible (except for the condition $0 \leq \delta < c$, which was conjectured to be unnecessary in \cite{Ellis}). However, the problem of obtaining a sharp stability result for sets not of size a power of 2, remained open.

\bigskip

\noindent \textbf{Our result}

\medskip

In this paper, we obtain the following stability result for the edge isoperimetric inequality in the discrete cube, which applies to families of arbitrary size and which is sharp up to an absolute constant factor.
\begin{thm}
\label{thm:main}
There exists an absolute constant $C>0$ such that
the following holds. If $\f\subset \pn$ and $\l \subset \p([n])$ is the initial segment of the lexicographic ordering with $|\l|=|\f|$, then there exists a family $\g \subset \p([n])$ weakly isomorphic to $\l$, such that
$$|\f \Delta \g| \leq C(|\partial \f| - |\partial \l|).$$
\end{thm}
This is sharp up to the value of the absolute constant $C$. In fact, we conjecture that Theorem \ref{thm:main} holds with $C=2$, due to the following example.
\begin{example}
Let $s,t,n$ be integers with $t \geq 2$ and $t+2 \leq s \leq n$, and let
$$\f = \f_{n,s,t} = \{S \subset [n]:\ [t] \subset S\} \cup \{S \subset [n]:\ [t-2] \cup \{t+1,t+2,\ldots,s\} \subset S\}.$$

\begin{comment}
$\f\subset\pn$ be the family such that
\begin{itemize}
\item The family $\f_{\left\{ 1,2\right\} }^{\left\{ 1,2\right\} }$ is
the subset of $\p\left(\left[n\right]\backslash\left\{ 1,2\right\} \right)$
consisting of the largest $2^{n-t-2}$ sets in the lexicographic ordering.
(I.e $\f_{\left\{ 1,2\right\} }^{\left\{ 1,2\right\} }=\s_{\left\{ 3,\ldots,t\right\} }$)
\item The families $\f_{\left\{ 1,2\right\} }^{\left\{ 1\right\} },\f_{\left\{ 1,2\right\} }^{\left\{ 2\right\} },\f_{\left\{ 1,2\right\} }^{\varnothing}$
consist of the largest $2^{n-s-2}$ sets in the lexicographic ordering. (I.e
$\f_{\left\{ 1,2\right\} }^{\left\{ 1,2\right\} }=\s_{\left\{ 3,\ldots,s\right\} }$)
\end{itemize}
\end{comment}

It is easy to see that
\begin{equation}
\label{eq:tightness-eg}
\min\{|\f \Delta \g|:\ \g \cong \l,\ \l \textrm{ is an initial segment of lex,}\ |\l| = |\f|\} =2(|\partial \f| - |\partial \l|),
\end{equation}
for each of the above families $\f$. We proceed to verify this when $s=n$ and $t=2$, i.e.\ for the family
$$\f_{n,n,2} = \{S \subset [n]: \{1,2\} \subset S\} \cup \{\{1,3,4,\ldots,n\},\{2,3,4,\ldots,n\},\{3,4,\ldots,n\}\};$$
the proof in the general case is similar. Let $\l \subset \pn$ be the initial segment of lex with $|\l| = |\f_{n,n,2}|$; then
$$\l = \{S \subset [n]: \{1,2\} \subset S\} \cup \{\{1,3,4,\ldots,n\},\{1,3,4,\ldots,n-1\},\{1,3,4,\ldots,n-2,n\}\}.$$
It is easy to check that $|\partial (\f_{n,n,2})| - |\partial \l|=2$. Now let $\g \subset \pn$ be any family weakly isomorphic to $\l$. Since $\l$ is contained in a codimension-1 subcube of $\pn$, so is $\g$. However, it is clear that $|\f_{n,n,2} \setminus \c| \geq 2$ for any codimension-1 subcube $\c$ of $\pn$. Hence, using the fact that $|\f_{n,n,2}| = |\g|$, we have
$$|\f_{n,n,2} \Delta \g| \geq 2|\f_{n,n,2} \setminus \g| \geq 2\min_{\c}|\f_{n,n,2} \setminus \c| \geq 4 = 2(|\partial (\f_{n,n,2})| - |\partial \l|),$$
where the minimum is over all codimension-1 subcubes $\c$ of $\pn$. Clearly, we have $|\f_{n,n,2} \Delta \l| = 4$. The last two facts imply (\ref{eq:tightness-eg}).
\end{example}

We note that the relation between $|\partial \f| - |\partial \l|$ and $|\f \Delta \g|$ in Theorem~\ref{thm:e} (which applies in the special case where $|\f|$ is a power of 2) is sharper than in Theorem~\ref{thm:main}, but the above example demonstrates that Theorem~\ref{thm:main} is sharp (up to an absolute constant factor) in its more general setting.

Instead of the Fourier-analytic techniques used in most previous works on isoperimetric stability, our techniques are purely combinatorial. As is often the case with theorems concerning $Q_n$, we prove Theorem \ref{thm:main} by induction on $n$, but the techniques we use in the inductive step are somewhat novel. The inductive step relies on an `intermediate' structure theorem (Proposition \ref{prop:implies Theorem main}) concerning the intersections of $\f$ with codimension-1 and codimension-2 subcubes, where $\f \subset \p([n])$ is a family with small edge boundary. This proposition is proved using some intricate combinatorial arguments, including shifting operators (a.k.a. `compressions'), and a detailed analysis of the {\em influences} of the family (see below).

\bigskip

\noindent \textbf{Related work}

\medskip

The edge boundary $\partial \f$ of a subset $\f \subset \pn $ is closely connected with the {\em influences} of $\f$.
%If $f:\{0,1\}^n \to \{0,1\}$ is a Boolean function, we define the $i$th {\em influence} of $f$ by
%$$\Inf_i[f] := \Pr_{x \in \{0,1\}^n} [f(x) \neq f(x + e_i)],$$
%where $e_i$ denotes the $i$th unit vector, $+$ denotes addition modulo 2, and the probability measure is the uniform one. In other words, the $i$th influence is the probability that when the input is chosen at random and then the value of the $i$th bit is flipped, the value of $f$ changes. The {\em total influence} of $f$ is the sum of its influences:
%$$I[f]: = \sum_{i=1}^{n} \Inf_i[f].$$
For $\f \subset \pn$, the $i$th influence of $\f$ is defined by
$$\Inf_i[\f] := |\{A \subset [n]: |\f \cap \{A, A\Delta \{i\}\}|=1\}|/2^n,$$
and the total influence of $\f$ is $I[\f] := \sum_{i=1}^n \Inf_i[\f]$. Note that $I[\f] = |\partial \f|/2^{n-1}$ --- the total influence of a set is none other than the size of its edge boundary, appropriately normalised.

It is natural to rephrase this definition in the language of Boolean functions. If $f:\{0,1\}^n \to \{0,1\}$, we define the $i$th influence of $f$ to be the probability that, if $x \in \{0,1\}^n$ is chosen uniformly at random and the $i$th entry of $x$ is flipped, the value of the function $f$ changes. There is of course a natural one-to-one correspondence between subsets of $\pn$ and Boolean functions on $\{0,1\}^n$: for each $\f \subset \pn$, we associate to $\f$ the Boolean function $\chi_\f:\{0,1\}^n \to \{0,1\}$ defined by $\chi_{\f}(x) = 1$ iff $\{i \in [n]:\ x_i =1\} \in \f$; the $i$th influence of $\f$ is then precisely the $i$th influence of $\chi_{\f}$.

Over the last thirty years, many results have been obtained on the influences of Boolean functions (and functions on more general product spaces), and have proved extremely useful in such diverse fields as theoretical computer science, social choice theory and statistical physics, as well as in combinatorics (see, e.g., the survey~\cite{Kalai-Safra}). One of the most useful such results (and one of the first major results on influences) is the seminal `KKL theorem' (Kahn, Kalai and Linial \cite{KKL}), which states that for any Boolean function $f:\{0,1\}^n \to \{0,1\}$ with $\mathbb{E}[f]=\mu$, there exists $i \in [n]$ such that
$$\Inf_i[f] \geq c_0 \mu(1-\mu) \frac{\log n}{n},$$
where $c_0$ is an absolute constant --- so a Boolean function of expectation $1/2$ has some coordinate of `fairly large' influence, viz., $\Omega((\log n)/n)$. (Note that if $\mathbb{E}[f]=1/2$, then Theorem \ref{thm:edge-iso} implies that $I[f] \geq 1$, so a naive averaging argument only supplies a coordinate of influence at least $1/n$.) The proof of Kahn, Kalai and Linial made crucial use of Fourier analysis on the discrete cube, together with the hypercontractive inequality due (independently) to Gross, Bonami and Beckner. Another very useful example is Friedgut's `Junta' theorem \cite{Friedgut98}:
\begin{thm}[Friedgut's Junta theorem]
There exists an absolute constant $C$ such that the following holds. Let $\epsilon >0$, and let $\f \subset \pn$. Then there exists $\g \subset \pn$ depending upon at most $2^{CI[\f]/\epsilon}$ coordinates, such that $|\f \Delta \g| \leq \epsilon 2^n$.
\end{thm}
Here, to be completely formal, if $\g \subset \pn$ we say that $\g$ {\em depends upon $k$ coordinates} if there exists $S \subset [n]$ with $|S|=k$, such that $(T \in \g) \Leftrightarrow (T \cap S \in \g)$ holds for all $T \subset [n]$. Friedgut's theorem implies that any $\f \subset \pn$ with bounded total influence (at most $K$, say), and with measure $|\f|/2^n$ bounded away from 0 and 1, can be closely approximated by a `junta' --- that is, by a family which depends upon a bounded number of coordinates (depending on $K$). Friedgut's proof in \cite{Friedgut98} uses Fourier analysis and hypercontractivity, in a similar way to the proof of the KKL theorem. (We remark that in \cite{fs}, Falik and Samorodnitsky gave different, more combinatorial proofs of the KKL theorem and of Friedgut's theorem, utilising martingales rather than Fourier analysis and hypercontractivity.)

For families $\f \subset \pn$ with measure $|\f|/2^n \in [\alpha,1-\alpha]$, Theorem~\ref{thm:edge-iso} implies that $I[\f] \geq 2\alpha \log_2(1/\alpha)$. Hence, Friedgut's theorem can be viewed as a structure theorem for families of measure bounded away from 0 and 1, whose total influence lies within a constant multiplicative factor of the minimum possible total influence. Similarly, Friedgut \cite{Friedgut-SAT}, Bourgain \cite{Bourgain99} and Hatami \cite{Hatami12} obtained structure theorems for `large' subsets of $\pn$ whose `biased' measure lies within a constant multiplicative factor of the minimum possible, and Kahn and Kalai \cite{KK06} stated several conjectures on `small' subsets of $\pn$ satisfying the same condition. The results of \cite{Bourgain99,Friedgut-SAT,Hatami12} are deep, with many important applications.

In contrast to the results of \cite{Bourgain99,Friedgut-SAT,Hatami12}, which describe the structure of families with total influence within a constant factor of the minimum, our Theorem \ref{thm:main} describes the structure of Boolean functions with total influence `very' close to the minimum. On the other hand, the structure we obtain is very strong --- namely, closeness to a genuinely extremal family.

\section{Outline of the Proof and Organization of the Paper}
\label{sec:outline}

The main step of our proof is showing an `intermediate' structural result (Proposition~\ref{prop:implies Theorem main} below) for families of small edge boundary (i.e., small total influence). Informally, this says that if $\f \subset \pn$ such that $|\f| \leq 2^{n-1}$ and $I\left[\f\right]\leq I\left[\l\right]+\epsilon$, where $\l$ the initial segment of the lexicographic ordering on $\pn$ of size $|\f|$ and $\epsilon$ is sufficiently small, then one of the following must hold.
\begin{enumerate}
\item[Case (1):] $\f$ is essentially contained in a subcube of codimension 1 (i.e., in a family depending upon just one coordinate), and the total influence of the part of $\f$ inside the subcube is `small'.
\item[Case (2):] $\f$ is essentially contained in a subcube of codimension 2 (i.e., in a family depending upon just two coordinates), and the total influence of the part of $\f$ inside that subcube is `small'.
\end{enumerate}
Once Proposition~\ref{prop:implies Theorem main} is established, the main theorem follows by a short induction on $n$ (Proposition \ref{prop:implies Theorem main} is needed for the inductive step). It is perhaps fortunate that, for the inductive step, it suffices to pass to subcubes of codimension at most 2. Interestingly, it does not suffice to pass to subcubes of codimension 1, as we explain in Section \ref{sec:Components-of-the}; this might, at first glance, deceive one into abandoning an inductive approach.

The proof of Proposition~\ref{prop:implies Theorem main} is divided into two parts. In the first part, we prove that if $|\f|$ is `sufficiently large' (specifically, if $|\f| \geq 2^{n-2}(1+c)$ for an absolute constant $c>0$), then $\f$ must satisfy~(1); this is the content of Proposition \ref{prop:mu large r large}. In the second part, we prove that if $|\f| < 2^{n-2}(1+c)$, then $\f$ must satisfy (2); this is the content of Proposition \ref{prop:mu small}. The harder part is the first one; the proof is (again) by induction on $n$, but with six ingredients that are outlined at the beginning of Section~\ref{sec:large}. Roughly speaking, we define a collection of `small alterations' which preserve the property of being a counterexample to Proposition \ref{prop:mu large r large}; applying a sequence of these small alterations, we reduce to the case where the family is sufficiently `well-behaved' for us to successfully apply the inductive hypothesis. (Note that a very similar technique was used in \cite{Keller Lifshitz}.) The second part uses the classical \emph{shifting} technique \cite{Daykin,EKR}: we first reduce to the case where $\f$ is monotone increasing; we then choose the coordinate of largest influence ($i$ say), and apply appropriate shifting operators to $\f$ to produce a family contained entirely within the codimension-1 subcube $\{S \subset [n]:\ i \in S\}$; passing to this subcube, we obtain a family of twice the measure of the original family; we then repeat this process until the family is large enough that we can apply Proposition \ref{prop:mu large r large} (from the first part).

An important component of the proof of Proposition \ref{prop:implies Theorem main} is a pair of `bootstrapping' lemmas showing that if $\f$ is `somewhat' close to being contained in a subcube of codimension 1 or 2, then it must be `very' close to that subcube. In order to prove these bootstrapping lemmas, we introduce the notion of \emph{fractional lexicographic families}, as a convenient technical tool.
\begin{comment} Fractional lexicographic families are functions from $\pn$ to $[0,1]$ for some $n$; they represent (non-fractional) families in $\mathcal{P}([n+m])$ (for some $m$) such that, when we partition $\mathcal{P}([n+m])$ into the $2^n$ codimension-$n$ subcubes with fixed coordinates $1,2,\ldots,n$, the intersection of $\mathcal{F}$ with each subcube is an initial segment of the lexicographical ordering (on that subcube). \end{comment} 
These allow us to analyse how the measure (or `mass') of a family of small total influence can be distributed between two disjoint codimension-1 subcubes (or between four disjoint codimension-2 subcubes); informally, this distribution cannot differ too much from in the extremal, lexicographically ordered family $\l$.

%One of the uses of the bootstrapping is claiming that if $\f$ is `far' from being contained in a subcube, then it cannot be transformed into a family `close' to be contained in a subcube by `small' modifications.

%\item A `shifting' procedure which allows to transform a family into a family contained in a dictatorship by a series of `small' modifications that do not increase the total influence.
\bigskip

\noindent \textbf{Organization of the paper}

\medskip

In Section~\ref{sec:preliminaries}, we introduce some definitions and notation, and present some basic facts on influences and shifting. In Section~\ref{sec:Components-of-the}, we reduce the main theorem to the intermediate structural result, Proposition \ref{prop:implies Theorem main}, discussed above. Fractional lexicographic families and their properties are studied in Section~\ref{sec:frac-lex}, and the bootstrapping lemmas are presented in Section~\ref{sec:Bootstrapping}.

The proof of Proposition \ref{prop:implies Theorem main} spans Sections \ref{sec:large}-\ref{sec:culmination}. The case of `large' families is covered in Section \ref{sec:large}, `small' families are dealt with in Section~\ref{sec:small}, and finally we combine these two cases to prove Proposition \ref{prop:implies Theorem main} in Section \ref{sec:culmination}. We conclude with some open problems in Section \ref{sec:conc}.

\section{Preliminaries}
\label{sec:preliminaries}
\subsection{Notation}

We equip $\pn$ with the uniform measure, denoted by $\mu$:
$$\mu(\mathcal{F}) = \frac{|\mathcal{F}|}{2^n} \quad \forall \f \subset \pn.$$
We write $S\sim\pn$ to mean that $S$ is chosen uniformly at randomly from $\pn$.

If $C\subset B \subset [n]$, and $\f \subset \pn$, we define the `sliced' family
$$\f_{B}^{C} := \{S \setminus C:\ S \in \f,\ S \cap B = C\} \subset \mathcal{P}([n] \setminus B).$$
Note that we view $\f_{B}^{C}$ as a subfamily of $\p\left(\left[n\right]\backslash B\right)$,
and so $\mu\left(\f_{C}^{B}\right)=\left|\f_{C}^{B}\right|/2^{n-|B|}$.

If $\mathcal{F},\mathcal{G} \subset \p([n])$, we say that $\f$ and $\g$ are {\em weakly isomorphic} if there exists an automorphism $\phi$ of $Q_n$ such that $\g = \phi(\f)$; in this case, we write $\f \cong \g$. To be completely formal and explicit, if $\pi \in \textrm{Sym}([n])$ and $S \subset \left[n\right]$, we write $\pi\left(S\right):=\left\{ \pi\left(i\right)\,:\, i\in S\right\}$,
and if $\f \subset \pn$, we write $\pi\left(\f\right):=\left\{ \pi\left(S\right)\,:\, S\in\f\right\}$. Families $\f,\g \subset \pn$ are said to be \emph{isomorphic} if there exists $\pi \in \Sym([n])$ such that $\g=\pi\left(\f\right)$. If $D \subset [n]$ and $\f \subset \p([n])$, we define $X_D(\f) = \{S \Delta D:\ S \in \f\}$. Families $\f,\g \subset \pn$ are weakly isomorphic iff there exist $\pi \in \Sym([n])$ and $D \subset [n]$ such that $\g=X_D(\pi\left(\f\right))$.

For $n \in \mathbb{N}$ and $2^n \mu \in \{0,1,\ldots,2^n\}$, we let $\l_{\mu,n}$ denote the initial segment of the lexicographic ordering on $\pn$ with measure $\mu$. We write $\mathbb{L}_{\mu,n}$ for the class of all families weakly isomorphic
to $\l_{\mu,n}$. When $n$ is understood, we will write these as $\l_{\mu}$ and $\mathbb{L}_{\mu}$, suppressing the subscript $n$. If $\f \subset \pn$, we write
$$\mu\left(\f\Delta\mathbb{L}_{\mu}\right) := \min\{\mu(\f \Delta \g):\ \g \cong \l_{\mu}\}.$$

We write $\mu_{i}^{-}=\mu_{i}^{-}(\f)$
for the measure $\mu(\f_{\{ i\} }^{\varnothing})$,
and we write $\mu_{i}^{+} = \mu_i^+(\f)$ for the measure $\mu(\f_{\{ i\} }^{\{ i\} })$.
By the isoperimetric inequality (Theorem \ref{thm:edge-iso}), we may write $I[\f_{\left\{ i\right\} }^{\left\{ i\right\} }]=I[\l_{\mu_{i}^{+}}]+\epsilon_{i}^{+}$,
where $\epsilon_{i}^{+}=\epsilon_{i}^{+}\left(\f\right) \geq 0$, and we
use the notations
\[
\mu_{i,j}^{++},\ \mu_{i,j}^{+-},\ \mu_{i,j}^{-+},\ \mu_{i,j}^{--},\ \epsilon_{i}^{-},\ \epsilon_{i,j}^{++},\ \epsilon_{i,j}^{+-},\ \epsilon_{i,j}^{-+},\ \epsilon_{i,j}^{--},
\]
defined similarly. For $B\subset\left[n\right],$ we write $\s_{B} := \{S \subset [n]:\ B \subset S\}$
for the subcube of all subsets of $[n]$ that contain $B$, and if
$C\subset B$, we write $\s_{B}^{C} := \{S \subset [n]:\ S \cap B = C\}$ for the subcube of all subsets
of $[n]$ that intersect $B$ on the set $C$. If $j \in [n]$, we write $\d_j := \{S \subset [n]:\ j \in S\}$ for the `dictatorship' consisting of all sets containing $j$.

%Abusing notation slightly, we will sometimes identify a family
%$\f\subset\pn$ with its indicator function $1_{\f} \colon\pn\to\left\{ 0,1\right\} $,
%i.e. we define $\f:\pn \to \{0,1\}$ by $\f\left(A\right)=1$ if $A$ is in $\f$, and $\f\left(A\right)=0$
%otherwise. 
We say that a family $\l\subset\p\left(\left[n\right]\right)$
is \emph{lexicographically ordered} if it is an initial segment of the lexicographic ordering on $\pn$. We say
that a family $\f\subset \pn$ is {\em monotone increasing} (or just {\em increasing}) if it is closed under taking supersets, i.e. whenever $A\subset  B\subset \left[n\right]$
and $A\in\f$, we have $B \in \f$.

\subsection{Influences}

Using the notation above, we may define the $i$th influence of a family
$\f \subset \pn$ by
\[
\Inf_{i}\left[\f\right] =\Pr_{A\sim\pn}[|\f \cap \{A, A\Delta \{i\}\}|=1].
\]
As mentioned in the introduction, we have
$$I\left[\f\right] =\sum_{i=1}^{n}\Inf_{i}\left[\f\right] = \frac{|\partial \f|}{2^{n-1}},$$
i.e. the total influence of $\f$ is the normalized edge boundary of $\f$. We may therefore restate our main theorem (Theorem \ref{thm:main}) as follows.

\begin{thm*}
There exists an absolute constant $C>0$ such that
the following holds. Let $\epsilon>0$, let $\f\subset\pn$ be a family of measure $\mu$, and suppose that $I\left[\f\right]\le I\left[\l_{\mu}\right]+\epsilon$.
Then there exists a family $\g \subset \pn$ weakly isomorphic to $\l_{\mu}$,
such that $\mu\left(\f\Delta\g\right)\le C\epsilon$.
\end{thm*}
\noindent (Note that the constant $C$ above is half the constant in the original statement.) It will be more convenient for us to work with the above reformulation.

If $\f \subset \p([n])$ and $i \in [n]$, we define the family of {\em $i$-pivotal
sets in $\f$} by
$$\mathcal{I}_{i}\left(\f\right):=\left\{ A\in\f\,:\, A\Delta \{i\}\notin\f\right\}.$$
Note that we have
$$\Inf_{i}\left[\f\right]=\frac{\left|\mathcal{I}_{i}\left(\f\right)\right|}{2^{n-1}}$$
for all $i \in [n]$.

The following lemma will be useful for relating the influence of a family $\f$ to the
influences of its slices.
\begin{lem}
\label{Lem: influences}If $S\subset\left[n\right]$ and $\f \subset \pn$, then
\begin{align*}
I\left[\f\right] =\frac{1}{2^{\left|S\right|}}\sum_{B\subset S}I\left[\f_{S}^{B}\right]+\frac{1}{2^{\left|S\right|-1}}\sum_{B\subset S}\sum_{i\in B}\mu\left(\f_{S}^{B}\Delta\f_{S}^{B\backslash\left\{ i\right\} }\right)  =\mathbb{E}_{B\sim\p\left(S\right)}I\left[\f_{S}^{B}\right]+\sum_{i\in S}\Inf_{i}\left[\f\right].
\end{align*}

\end{lem}
\noindent The proof is straightforward, and we omit it.
 
\subsection{Shifting}

The following shifting operator $\mathcal{S}_{S,T}$ was introduced by Erd\H{o}s,
Ko and Rado \cite{EKR} in the case $\left|S\right|=\left|T\right|=1$; for larger values of $|S|$ or $|T|$, it was introduced by Daykin \cite{Daykin}.
\begin{defn}
Let $n\in\mathbb{N}$, let $\f\subset\mathcal{P}\left(\left[n\right]\right),$
and let $S,T\subset\left[n\right]$ with $S \cap T= \emptyset$. For a set
$A\in\f$, we define
$$\mathcal{S}_{ST}\left(A\right):= \begin{cases} \left(A\backslash S\right)\cup T & \mbox{ if } S \subset A,\ A \cap T = \emptyset \mbox{ and } \left(A\backslash S\right)\cup T\notin\f,\\
A &  \mbox{ otherwise.}\end{cases}$$
We define $\s_{ST}\left(\f\right):=\left\{ \s_{ST}\left(A\right)\,:\, A\in\f\right\}$.
\end{defn}
Observe that $\s_{ST}\left(\f\right)$ is the family $\g \subset \pn$ such that
$\g_{S\cup T}^{B}=\f_{S\cup T}^{B}$ for any $B\ne S,T$, such
that $\g_{S\cup T}^{S}=\f_{S\cup T}^{S}\cap\f_{S\cup T}^{T}$, and
such that $\g_{S\cup T}^{T}=\f_{S\cup T}^{S}\cup\f_{S\cup T}^{T}$.

These shifting operators are known to be a very useful tool in extremal combinatorics. They were used by Frankl \cite{Frankl86} to obtain stability results for the Erd\H{o}s-Ko-Rado theorem \cite{EKR}, and were recently applied by the authors in \cite{Keller Lifshitz} to obtain a stability result for the Ahlswede-Khachatrian theorem~\cite{AK}, thus proving a conjecture of Friedgut~\cite{Friedgut08}. A major part of our argument is based on the method of \cite{Keller Lifshitz}.

The following lemma says that if a family $\f$ is stable under `lower-order' shifts, then a shifting operator cannot increase the total influence of $\f$.
\begin{lem}
\label{lem:shifting dec inf}Let $\f\subset\pn$, and let $S,T\subset\left[n\right]$
with $S \cap T = \emptyset$ and $\left|S\right|\ge\left|T\right|$. Suppose
that $\s_{S'T}\left(\f\right)=\f$ for each $S'\subset S$ with
$|S'| = \left|S\right|-1$. Then $I\left[\s_{ST}\left(\f\right)\right]\le I\left[\f\right]$. \end{lem}
\begin{proof}
Write $\g=\s_{ST}\left(\f\right)$. By Lemma \ref{Lem: influences},
we have
\[
I\left[\f\right]=\sum_{i\in\left[n\right]\backslash (S\cup T)}\Inf_{i}\left[\f\right]+\mathbb{E}_{B\sim\p\left(\left[n\right]\backslash (S\cup T)\right)}I\left[\f_{\left[n\right]\backslash (S\cup T)}^{B}\right]
\]
and
\begin{align*}
I\left[\g\right] & =\sum_{i\in\left[n\right]\backslash (S\cup T)}\Inf_{i}\left[\g\right]+\mathbb{E}_{B\sim\p\left(\left[n\right]\backslash (S\cup T)\right)}I\left[\g_{\left[n\right]\backslash (S\cup T)}^{B}\right].
\end{align*}
To prove the claim, it suffices to show that for any family $\f \subset \pn$ and any $i \notin S\cup T$, we have $\Inf_{i}\left[\f\right]\ge\Inf_{i}\left[\s_{ST}\left(\f\right)\right]$, and that for any family $\f\subset\p\left(S\cup T\right)$ such that $\s_{S'T}\left(\f\right)=\f$ for each $S'\subset S$ with
$|S'| = \left|S\right|-1$, we have $I\left[\s_{ST}\left(\f\right)\right]\le I\left[\f\right]$. The verification
of the former assertion is straightforward, and we leave it to the reader. 

To prove the latter assertion, we may assume that $\s_{ST}\left(\f\right) \neq \f$; then $S \in \f$, $T \notin \f$ and $\s_{ST}\left(\f\right) = (\f \setminus \{S\}) \cup \{T\}$. Note that $S' \notin \f$ for all $S'\subset S$
with $|S'|=|S|-1$, and that $T' \in \f$ for any $T'\supset T$
with $|T'|=|T|+1$. (Indeed, if $S'\subset S$ with $|S'| = |S|-1$ and $S' \in \f$, then we have $T=\s_{S'T}\left(S'\right)\in\f$, contradicting our assumption. Similarly, for each $T'\supset T$ with $|T'|=|T|+1$, we have $T'=\s_{\left(\left[S\cup T\right]\backslash T'\right)T}\left(S\right)\in\f$.) Therefore, $|\partial \left(\s_{ST}\left(\f\right)\right)| \leq |\partial \f| - 2(|S|-|T|) \leq |\partial \f|$, as required.
\end{proof}

We also need the following well-known lemma on the so-called `monotonization
operators' $\s_{\varnothing \{i\}}$ (see e.g.\ \cite{KKL}).

\begin{lem}
\label{lem:mon dec inf}
Let $i \in [n]$ and let $\f\subset\pn$. Then
\[
\Inf_{j}\left[\s_{\varnothing \{i\}}\left(\f\right)\right]\le\Inf_{j}\left[\f\right]
\]
for each $j\in [n]$, and $I\left[\s_{\varnothing \left\{ i\right\} }\left(\f\right)\right]\le I\left[\f\right]$.
\end{lem}

\section{\label{sec:Components-of-the} Reduction of Theorem \ref{thm:main} }

%Write $\mu=2^{-j}+r$ , let $\epsilon>0$, and let $\f\subset\pn$
%be a family with measure $\mu\left(\f\right)=\mu$, and with $I\left[\f\right]=I\left[\l_{\mu}\right]+\epsilon$.
%The general idea of the proof of Theorem \ref{thm:main} is to show
%that $\f$ is essentially contained in some subcube $\s_{C}^{B}$,
%such that $I\left[\f_{C}^{B}\right]$ is small. This enables us to complete
%the proof by induction. 
In this section we reduce Theorem \ref{thm:main} to the following proposition.
\begin{prop}
\label{prop:implies Theorem main}There exist absolute constants $c_{1},c_{2}>0$
such that the following holds. Let $0 < \mu\le\frac{1}{2}$, let $0 \leq \epsilon\le c_{1}\mu$,
and let $\f\subset\pn$ be a family with $\mu\left(\f\right)=\mu$ and $I\left[\f\right]\leq I\left[\l_{\mu}\right]+\epsilon$. Then there exists a family $\g$ weakly isomorphic to $\f$ such that one of the following holds.
\begin{itemize}
\item Case (1): $c_2 \mu_{1}^{-}\left(\g\right)+\frac{1}{2}\epsilon_{1}^{+}\left(\g\right)\le\epsilon$, or else
\item Case (2): $c_{2}\mu\left(\g\backslash\s_{\left\{ 1,2\right\} }\right)+\frac{1}{4}\epsilon_{1,2}^{++}\left(\g\right)\le\epsilon$.
\end{itemize}
\end{prop}
Intuitively, Proposition \ref{prop:implies Theorem main} says that
there is a family $\g$ weakly isomorphic to $\f$, such that one
of the following holds: either (1) $\g$ is essentially contained in the
dictatorship $\d_{1}$, and the `essential part' $\g_{\left\{ 1\right\} }^{\left\{ 1\right\} }$
has small total influence, or (2) $\g$ is essentially contained in the subcube $\s_{\left\{ 1,2\right\} }$, and the `essential part' $\g_{\left\{ 1,2\right\} }^{\left\{ 1,2\right\} }$ has small total influence.

\begin{remark}
Case (1) is in a sense the `simpler' case, and it is natural to ask whether Case (2) can be removed, but it cannot. Indeed, for any $c_1,c_2>0$, if $t$ is sufficiently large depending on $c_1$ and $c_2$, then the family
$$\f = \{S \subset [n]:\ \{1,2\} \subset S,\ S\cap \{3,4,\ldots,t\} \neq \emptyset\} \cup \{S \subset [n]:\ \{3,4,\ldots,t\} \subset S\}$$
satisfies the hypotheses of Proposition \ref{prop:implies Theorem main}, but Case (1) does not occur for $\f$. To see this, let $t \geq 4$ and suppose that Case (1) occurs for $\f$. Let $\mu := \mu(\f) = \tfrac{1}{4}+2^{-(t-1)}$ and let $\epsilon = I[\f] - I[\l_{\mu}]$. Let $\g$ be a family weakly isomorphic to $\f$ such that 
$$c_2 \mu_{1}^{-}\left(\g\right)+\tfrac{1}{2}\epsilon_{1}^{+}\left(\g\right)\le\epsilon.$$
Then there exists $\pi \in \Sym([n])$ and $D \subset [n]$ such that $\g = X_{D}(\pi(\f))$. We have
$$\l_{\mu} = \{S \subset [n]:\ \{1,2\} \subset S\} \cup \{S \subset [n]:\ 1 \in S,\ 2 \notin S,\ \{3,4,\ldots,t-1\} \subset S\}.$$
It is easy to check that $I[\f] = 1+8t\cdot 2^{-t} -24 \cdot 2^{-t}$ and $I[\l_\mu] = 1+4t\cdot 2^{-t} - 12 \cdot 2^{-t}$, and therefore
$$\epsilon = I[\f]-I[\l_{\mu}] = 4t\cdot 2^{-t} - 12 \cdot 2^{-t} \leq c_1 \mu$$
if $t$ is sufficiently large depending on $c_1$.
Observe that
$$\mu_1^{+}(\f) = \mu_2^{+}(\f) = \tfrac{1}{2}, \quad \min\{\mu_j^+(\f),\mu_j^{-}(\f)\} = \tfrac{1}{4}-2\cdot 2^{-t} \geq \tfrac{1}{8} \ \forall j \geq 3,$$
and that $\f$ is invariant under permuting the coordinates 1 and 2. Hence, if $t$ is large enough that $c_2 \cdot \tfrac{1}{8} >\epsilon = 4t\cdot 2^{-t} - 12 \cdot 2^{-t}$, then we may assume that $D = \emptyset$ and $\pi = \Id$, i.e.\ $\g$ can be obtained from $\f$ without flipping or permuting any coordinates, so $\g=\f$. Hence,
\begin{equation} \label{eq:1-is-good} c_2 \mu_{1}^{-}\left(\f\right)+\tfrac{1}{2}\epsilon_{1}^{+}\left(\f\right)\leq \epsilon = 4t\cdot 2^{-t} - 12 \cdot 2^{-t}.\end{equation}
We have
\begin{align*} \f_{\{1\}}^{\{1\}} = &\ \{S \subset [n] \setminus \{1\}:\ 2 \in S,\ S \cap \{3,4,\ldots,t\} \neq \emptyset\} \cup\{S \subset [n] \setminus \{1\}:\ 2 \notin S,\ \{3,4,\ldots,t\} \subset S\}
\end{align*}
and therefore $I[\f_{\{1\}}^{\{1\}}] = 1+8t \cdot 2^{-t} - 24 \cdot 2^{-t}$. Since $\mu(\f_{\{1\}}^{\{1\}}) = \mu_1^{+}(\f) = \tfrac{1}{2}$ and $I[\l_{1/2}] = 1$, we have
$$\epsilon_1^+(\f) = I[\f_{\{1\}}^{\{1\}}] - I[\l_{1/2}] = 1+8t \cdot 2^{-t} - 24 \cdot 2^{-t} - 1 = 8t \cdot 2^{-t} - 24 \cdot 2^{-t}.$$
Finally, we have $\mu_1^{-}(\f) = 4 \cdot 2^{-t}$. Substituting the latter two facts into (\ref{eq:1-is-good}) yields
$$c_2 \cdot 4 \cdot 2^{-t} + \tfrac{1}{2}(8t \cdot 2^{-t} - 24 \cdot 2^{-t}) \leq \epsilon = 4t\cdot 2^{-t} - 12 \cdot 2^{-t},$$
a contradiction.
\end{remark}

We now show how to deduce Theorem \ref{thm:main} from Proposition \ref{prop:implies Theorem main}.
We state Theorem \ref{thm:main} again below (in the influence form), for the convenience of the
reader.
\begin{thm*}
There exists an absolute constant $C>0$ such that the following
holds. Let $\f\subset\pn$ be a family of measure $\mu\left(\f\right)=\mu$,
and suppose that $I\left[\f\right]\le I\left[\l_{\mu}\right]+\epsilon$.
Then there exists a family $\g \subset \pn$ weakly isomorphic to $\l_{\mu}$,
such that $\mu\left(\f\Delta\g\right)\le C\epsilon$.\end{thm*}
\begin{proof}
We prove the theorem by induction on $n$. If $n=1$, then $\f$ itself
is weakly isomorphic to $\l_{\mu}$. Let $n\ge2$, and assume the statement of the theorem holds for smaller values of $n$.

We now make several reductions. Firstly, we note that the theorem
holds for $\f$ if and only if it holds for its complement $\f^{c}$, since the complement of a lexicographically ordered family is weakly isomorphic to a lexicographically ordered family. Thus, we may assume w.l.o.g.\ that $\mu\left(\f\right)\le\frac{1}{2}$. Secondly, note that the conclusion of the theorem holds trivially if $C\epsilon \geq 2\mu$. So we may assume throughout that $\epsilon < \frac{2\mu}{C}$.
\begin{comment}
Secondly, we also note the theorem holds for $\f$ if and only if
it holds for some family weakly isomorphic to it. Hence, we assume
w.l.o.g.\ that $\Inf_{1}$ is the largest influence, and that $\mu\left(\f_{\left\{ 1\right\} }^{\varnothing}\right)\le\mu\left(\f_{\left\{ 1\right\} }^{\left\{ 1\right\} }\right)$.
Finally, we assume w.l.o.g.\ that the minimum of $\mu(\f_{\left\{ 1\right\} }^{\left\{ 1\right\} }\Delta\h)$
over $\h$ weakly isomorphic to $\l_{\mu\left(\f_{\left\{ 1\right\} }^{\left\{ 1\right\} }\right)}$
is obtained where $\h$ is a lexicographically ordered family with
respect to the usual ordering $2\le3\le\cdots\le n$. After making
the above reductions we may take $\g=\l_{\mu}$ and we show that $\f$
is close to $\l_{\mu}$.
\end{comment}

Provided $C \geq 2/c_1$, we have $\epsilon \leq c_1\mu$, so either Case (1) or Case (2) of Proposition \ref{prop:implies Theorem main} occurs. First suppose that Case (1) occurs. By replacing $\f$ by a family $\g$ weakly isomorphic to $\f$ if necessary, we may assume that
\begin{equation} \label{eq:replace-assumption} c_2 \mu_{1}^{-}\left(\f\right)+\tfrac{1}{2}\epsilon_{1}^{+}\left(\f\right)\le\epsilon.\end{equation}
Assume w.l.o.g.\ that the minimum of $\mu(\f_{\left\{ 1\right\} }^{\left\{ 1\right\} }\Delta\h)$
over all families $\h$ weakly isomorphic to $\l_{\mu(\f_{\left\{ 1\right\} }^{\left\{ 1\right\} })}$
is attained where $\h = \l_{\mu(\f_{\left\{ 1\right\} }^{\left\{ 1\right\} })}$, i.e.\ where $\h$ is a lexicographically ordered family with respect to the usual ordering $2\le3\le\cdots\le n$.

Note that
\begin{equation}
\mu\left(\f\Delta\l_{\mu}\right)=2\mu\left(\f\backslash\l_{\mu}\right)=\mu\left(\f_{\left\{ 1\right\} }^{\left\{ 1\right\} }\backslash \left(\l_{\mu}\right)_{\{1\}}^{\{1\}}\right)+\mu\left(\f_{\left\{ 1\right\} }^{\varnothing}\right).\label{eq:mu(F)}
\end{equation}
The induction hypothesis, and our assumption above on the families $\h$, imply that
\begin{equation}
\mu\left(\f_{\left\{ 1\right\} }^{\left\{ 1\right\} }\backslash \left(\l_{\mu}\right)_{\{1\}}^{\{1\}}\right)\le\tfrac{1}{2}\mu\left(\f_{\left\{ 1\right\} }^{\left\{ 1\right\} }\Delta\mathcal{L}_{\mu_{1}^{+}}\right)\le\tfrac{1}{2}C\epsilon_1^{+}.\label{eq: mu(Fone)}
\end{equation}
 Rearranging (\ref{eq:replace-assumption}), we have
\begin{equation}
\epsilon_{1}^{+} \le 2(\epsilon-c_2 \mu_1^{-}).\label{eq:epsilons}
\end{equation}
Putting together (\ref{eq:mu(F)}), (\ref{eq: mu(Fone)}) and (\ref{eq:epsilons}), we obtain
\[
\mu\left(\f\Delta\l_{\mu}\right)\le C(\epsilon-c_2\mu_1^{-})+\mu_1^{-}\le C\epsilon,
\]
provided $C\ge\frac{1}{c_2}$. This completes the proof in Case (1).

Suppose now that Case (2) occurs. Replacing $\f$ by a family $\g$ weakly isomorphic to it, we may assume that
\begin{equation}
c_{2}\mu\left(\f\backslash\s_{\left\{ 1,2\right\} }\right)+\tfrac{1}{4}\epsilon_{1,2}^{++}\left(\f\right)\le\epsilon.\label{eq:4}
\end{equation}

Assume w.l.o.g.\ that the minimum of $\mu(\f_{\left\{ 1,2\right\} }^{\left\{ 1,2\right\} }\Delta\h)$
over all families $\h$ weakly isomorphic to $\l_{\mu(\f_{\left\{ 1,2\right\} }^{\left\{ 1,2\right\} })}$
is attained where $\h = \l_{\mu(\f_{\left\{ 1,2\right\} }^{\left\{ 1,2\right\} })}$, i.e.\ where $\h$ is a lexicographically ordered family with respect to the usual ordering $3\le\cdots\le n$.

We now have
\begin{align}
\mu\left(\f\Delta\mathcal{L}_{\mu}\right) & \le\tfrac{1}{4}\mu\left(\f_{\left\{ 1,2\right\} }^{\left\{ 1,2\right\} }\Delta\mathcal{L}_{\mu_{1,2}^{++}\left(\f\right)}\right)+2\mu\left(\f\backslash\s_{\left\{ 1,2\right\} }\right).\label{eq:5}
\end{align}
By the induction hypothesis, we have
\begin{equation}
\mu\left(\f_{\left\{ 1,2\right\} }^{\left\{ 1,2\right\} }\Delta\mathcal{L}_{\mu_{1,2}^{++}\left(\f\right)}\right)\le C\epsilon_{1,2}^{++}.\label{eq:6}
\end{equation}

Putting together (\ref{eq:4}), (\ref{eq:5}) and (\ref{eq:6}), we obtain
\[
\mu\left(\f\Delta\mathcal{L}_{\mu}\right)\le C\left(\epsilon-c_{2}\mu\left(\f\backslash\s_{\left\{ 1,2\right\} }\right)\right)+2\mu\left(\f\backslash\s_{\left\{ 1,2\right\} }\right)\le C\epsilon,
\]
where the last inequality holds provided $C\ge\frac{2}{c_{2}}$. This completes the proof.
\end{proof}

\section{Fractional lexicographic families and their properties}
\label{sec:frac-lex}

%\subsection{Fractional lexicographic families}

A {\em fractional lexicographic family of order $n$} is a function $\f\colon\p\left(\left[n\right]\right)\to\mathbb{D}$, where $\mathbb{D} = \{\tfrac{b}{2^a}:\ a \in \mathbb{N},\ b \in \{0,1,\ldots,2^a\}\}$ denotes the set of dyadic rationals between 0 and 1. Intuitively, a fractional lexicographic family $\f\colon\p\left(\left[n\right]\right)\to\mathbb{D}$ represents a (non-fractional) family $\g\subset\p\left([n + m]\right)$, for some $m \in \mathbb{N}$,
such that $\g_{\left[n\right]}^{B}$ is a lexicographically ordered
family of measure $\f\left(B\right)$, for each $B\subset [n]$. Formally, if $\f\colon\pn\to\mathbb{D}$, we choose any $m \in \mathbb{N}$ such that $2^m \f(\pn) \subset \mathbb{Z}$, and associate to $\f$ the family\[
\f_{\text{ass}}\subset\p\left([n+m]\right)
\]
such that $(\f_{\text{ass}})_{[n]}^{B}$ is the lexicographically ordered family of measure $\f(B)$, for each $B \subset [n]$. If $\f$ is a fractional lexicographic family, then by a slight abuse of notation we define $\mu(\f)$, $I[\f]$, $\mu_i^+(\f)$ and $\mu_i^-(\f)$ (for $i \in [n]$) to be the corresponding quantities for the associated family $\f_{\text{ass}}\subset\p\left(\left[m+n\right]\right)$; it is easy to see that these are independent of the choice of $m$, provided we define $\Inf_i[\f_{\text{ass}}] = 0$ for all $i > n+m$.

The usefulness of fractional lexicographic families comes from the fact that
inductive arguments enable us to reduce statements about general families to
statements about fractional lexicographic families of order $n$, for small $n$. Specifically,
we need a thorough analysis of the case $n=1$ and the case $n=2$. These statements encapsulate the idea that families of small total influence can only have their measure split between two disjoint codimension-1 subcubes (or between four disjoint codimension-2 subcubes) in certain ways. The proofs of the statements  are technical and the reader is advised (at least at first reading) to read the statements of the lemmas without going into their proofs.

\subsection{Properties of fractional lexicographic families of order 1}

%In this section, we state some useful facts on fractional lexicographic families of orders 1 and 2. The proofs are somewhat technical, so are deferred to the Appendix. 
If $0 \leq \mu^-,\mu^+\leq 1$, we denote by $\l_{\mu^{-},\mu^{+}}$ the fractional lexicographic family $\l\colon\left\{ \varnothing,\left\{ 1\right\} \right\} \to\left[0,1\right]$ of order 1, with $\l\left(\left\{ 1\right\} \right)=\mu^{+}$ and $\l\left(\varnothing\right)=\mu^{-}$.

Let $\mu=2^{-j}+r$, where $j \geq 2$ and $0<r\le2^{-j}$. Observe that
\begin{equation}\label{Eq:Aux1}
\mu_{i}^{-}\left(\l_{\mu}\right) \begin{cases} = 0 \mbox{ if } i\le j-1,\\ = 2r \mbox{ if } i=j,\\
\geq \frac{1}{2}\mu\mbox{ if }i \geq j+1.\end{cases}
\end{equation}
The next lemma says roughly that if a fractional lexicographic family $\l=\l_{\mu^{-},\mu^{+}}$ of order 1 has $0 < \mu^- \leq r$, then $I\left[\l_{\mu^{-},\mu^{+}}\right]$ is somewhat large.

\begin{lem}
\label{lem:1-lexicographical}Let $j \geq 2$, let $0<r\le2^{-j}$, let $\mu=2^{-j}+r$, and let $0 \leq \mu^{-}\le\mu^{+} \leq 1$ with
$\frac{\mu^{-}+\mu^{+}}{2}=\mu$. If $\mu^{-}\le r$, then $I\left[\l_{\mu^{-},\mu^{+}}\right]\ge I\left[\l_{\mu}\right] + 2\mu^{-}$.

If instead, $3r\le\mu^{-}\le \tfrac{1}{2}\mu$, then $I\left[\l_{\mu^{-},\mu^{+}}\right]\ge I\left[\l_{\mu}\right] + \tfrac{2}{3}\mu^-$.
\end{lem}

In order to prove the lemma, we need the following preparatory claim.
\begin{claim}
\label{Claim: r large mu large}Let $\mu=\frac{\mu^{-}+\mu^{+}}{2}=\frac{1}{4}+r$, where $0 < r \leq 1/4$ and $0 \leq \mu^-,\mu^+ \leq 1$. Suppose that $\mu^{-}\le r$. Then
\[
I\left[\l_{\mu^{-},\mu^{+}}\right]\ge I\left[\l_{\mu}\right]+2\mu^{-}.
\]
\end{claim}
\begin{proof}
Write $\l=\l_{\mu^{-},\mu^{+}}$. Then we may view $\l$ as a fractional lexicographical
family on $\p\left(\left[2\right]\right)$ such that $\l\left(\left\{ 1,2\right\} \right)=1$,
$\l\left(\varnothing\right)=\varnothing$, $\l\left(\left\{ 2\right\} \right)=2\mu^{-},$
and $\l\left(\left\{ 1\right\} \right)=2\mu^{+}-1$. Using Lemma
\ref{Lem: influences} (and writing $\l$ in place of $\l_{\textrm{ass}}$ in the first line of (\ref{eq:computation 1}) below), we have
\begin{align}
I\left[\l\right] & =\tfrac{1}{4}\sum_{B\subset\left\{ 1,2\right\} }I[\l_{\left\{ 1,2\right\} }^{B}]+\tfrac{1}{2}\left(\mu(\l_{\left\{ 1,2\right\} }^{\left\{ 1,2\right\} })-\mu(\l_{\left\{ 1,2\right\} }^{\left\{ 1\right\} })\right) + \tfrac{1}{2}\left(\mu(\l_{\left\{ 1,2\right\} }^{\left\{ 1,2\right\} })-\mu(\l_{\left\{ 1,2\right\} }^{\left\{ 2\right\} })\right) \label{eq:computation 1}\\
&\quad +\tfrac{1}{2}\left(\mu(\l_{\left\{ 1,2\right\} }^{\left\{ 1\right\} })-\mu(\l_{\left\{ 1,2\right\} }^{\varnothing})\right)+\tfrac{1}{2}\left(\mu(\l_{\left\{ 1,2\right\} }^{\left\{ 2\right\} })-\mu(\l_{\left\{ 1,2\right\} }^{\varnothing})\right)\nonumber \\
 & =\tfrac{1}{4}\left(I\left[\l_{2\mu^{-}}\right]+I\left[\l_{2\mu^{+}-1}\right]\right)+\l(\{1,2\})-\l(\varnothing).\nonumber
\end{align}
 Similarly, $\m := \l_{\mu}$ may viewed as a lexicographical family on $\p\left(\left[2\right]\right)$
with $\m\left(\varnothing\right)=\varnothing$, $\m\left(\left\{ 2\right\} \right)=\varnothing$,
$\m\left(\left\{ 1,2\right\} \right)=1$, and $\m\left(\left\{ 1\right\} \right)=4\mu-1$.
So as in (\ref{eq:computation 1}), we have
\begin{equation}
I\left[\m\right]=\tfrac{1}{4}I\left[\l_{4\mu-1}\right]+\l\left(\left\{ 1,2\right\} \right)-\l\left(\varnothing\right).\label{eq:computation1 prime}
\end{equation}
Putting (\ref{eq:computation 1}) and (\ref{eq:computation1 prime})
together, we have
\begin{align}
I\left[\l\right]-I\left[\m\right] & =\tfrac{1}{4}\left(I\left[\l_{2\mu^{-}}\right]+I\left[\l_{2\mu^{+}-1}\right]-I\left[\l_{4\mu-1}\right]\right).\label{eq:comp2}
\end{align}
We now consider the families $\f_{1}:=\l_{\frac{4\mu-1}{2}}$
and $\f_{2}:=\l_{2\mu^{-},2\mu^{+}-1}$. The isoperimetric
inequality implies that
\begin{equation}
I\left[\f_{2}\right]\ge I\left[\f_{1}\right].\label{eq:non comp1}
\end{equation}
Let us compute the influences of $\f_{1}$ and $\f_{2}$. We have
\begin{equation}
I\left[\f_{1}\right]=\tfrac{1}{2}I\left[\l_{4\mu-1}\right] + 4\mu-1\label{eq:comp 3}
\end{equation}
 and
\begin{equation}
I\left[\f_{2}\right]=\tfrac{1}{2}I\left[\l_{2\mu^{-}}\right]+\tfrac{1}{2}I\left[\l_{2\mu^{+}-1}\right]+2\mu^{+}-1-2\mu^{-}.\label{eq:comp4}
\end{equation}
 (Here, we used the fact that $2\mu^{+}-1 = 4\mu-2\mu^{-}-1=4r-2\mu^{-}\ge2\mu^{-}$.)

Combining (\ref{eq:comp2})-(\ref{eq:comp4}) yields
\begin{align*}
I\left[\l\right]-I\left[\m\right] & = \tfrac{1}{4}\left(I\left[\l_{2\mu^{-}}\right]+I\left[\l_{2\mu^{+}-1}\right]-I\left[\l_{4\mu-1}\right]\right)\\
 & =\tfrac{1}{2}\left(I\left[\f_{2}\right]-\left(2\mu^{+}-1-2\mu^{-}\right)\right) -\tfrac{1}{2}\left(I\left[\f_{1}\right]-\left(2\mu^{+}+2\mu^{-}-1\right)\right)\\
 & \ge2\mu^{-},
\end{align*}
 as required.
\end{proof}

We can now prove Lemma \ref{lem:1-lexicographical}.
\begin{proof}[Proof of Lemma \ref{lem:1-lexicographical}.]
We prove the first statement by induction on $j$. Suppose that $\mu^{-}\le r$. Since $\mu^{+}=2\mu-\mu^{-}>\left(\frac{1}{2}\right)^{j-1}$,
we have $j \geq 2$. Claim \ref{Claim: r large mu large}
implies the base case $j=2$. Let $j \geq 3$, and assume the statement holds for smaller values of $j$. Since $\mu \leq 1/8$, $(\l_{\mu^{-},\mu^{+}})_{\textrm{ass}}$ is contained
in the dictatorship $\d_{2}$. By Lemma \ref{Lem: influences}, we
have
\begin{align}
I\left[\l_{\mu^{-},\mu^{+}}\right] =\tfrac{1}{2}I\left[\left(\l_{\mu^{-},\mu^{+}}\right)_{\left\{ 2\right\} }^{\left\{ 2\right\} }\right]+\tfrac{1}{2}I\left[\left(\l_{\mu^{-},\mu^{+}}\right)_{\left\{ 2\right\} }^{\left\{ \varnothing\right\} }\right]+\Inf_{2}\left[\l_{\mu^{-},\mu^{+}}\right] =\tfrac{1}{2}I\left[\l_{2\mu^{-},2\mu^{+}}\right]+2\mu.\label{eq:8}
\end{align}
Similarly,
\begin{equation}
I\left[\l_{\mu}\right]=\tfrac{1}{2}I\left[\l_{2\mu}\right]+2\mu.\label{eq:9}
\end{equation}
 The induction hypothesis implies that
\begin{equation}
I\left[\l_{2\mu^{-},2\mu^{+}}\right] \geq I\left[\l_{2\mu}\right] + 4\mu^-.\label{eq:10}
\end{equation}
Combining (\ref{eq:8}), (\ref{eq:9}), and (\ref{eq:10}), we obtain $I\left[\l_{\mu^{-},\mu^{+}}\right] \geq I\left[\l_{\mu}\right] + 2\mu^-$, as required.

Now suppose that $3r \leq \mu^- \leq \tfrac{1}{2}\mu$. We proceed again by induction on $j$. First suppose $j=1$. Note that $I\left[\l_{1-\mu^{+},1-\mu^{-}}\right]=I\left[\l_{\mu^{-},\mu^{+}}\right]$,
and that $\mu\left(\l_{1-\mu^{+},1-\mu^{-}}\right)=\frac{1}{2}-r=\frac{1}{4}+\left(\frac{1}{4}-r\right)$.
We also have
\[
\mu^{+} = 2\mu-\mu^{-}\ge \tfrac{3}{4} + \tfrac{3}{2}r.
\]
Hence,
\[
1-\mu^{+}\leq \tfrac{1}{4}-\tfrac{3}{2}r\le\tfrac{1}{4}-r.
\]
Claim \ref{Claim: r large mu large}
implies that
\begin{align*}
I\left[\l_{\mu^{-},\mu^{+}}\right] =I\left[\l_{1-\mu^{+},1-\mu^{-}}\right] \ge I\left[\l_{1-\mu}\right]+ 2\left(1-\mu^{+}\right) & = I\left[\l_{\mu}\right]+ 2\left(1-2\mu+\mu^{-}\right) \\
& =I\left[\l_{\mu}\right] + 2\left(\mu^{-}-2r\right) \ge I\left[\l_{\mu}\right] + \tfrac{2}{3} \mu^-,
\end{align*}
 as desired. The inductive step is almost exactly the same as in the previous
case, relying on the fact that $(\l_{\mu^{-},\mu^{+}})_{\textrm{ass}}$ is contained
in the dictatorship $\d_{2}$.
\end{proof}

\subsection{Properties of fractional lexicographic families of order 2}

%We now consider fractional lexicographic families of order 2. 
Let $\mu=2^{-j}+r$, where $j \geq 2$ and $0<r\le2^{-j}$, and let $\l\colon\p\left(\left[2\right]\right)\to\left[0,1\right]$ be
a fractional lexicographic family of order 2 and measure $\mu$. The following lemma says that if both $\mu_{1}^{-}\left(\l\right)$
and $\mu_{2}^{-}\left(\l\right)$ are `somewhat close' to $2r$ (which, by (\ref{Eq:Aux1}), is the value of $\mu_j(\l_{\mu})$), then $\l$ has `somewhat large' total influence. 

\begin{lem}\label{lem:2-lex}
There exists an absolute constant $c>0$ such that the following holds. Let $\mu=2^{-j}+r$, where $j \in \mathbb{N}$ and $r\le2^{-j}$. Let $\l\colon\p\left(\left[2\right]\right)\to\left[0,1\right]$
be a fractional lexicographic family of order 2 and measure $\mu$. Suppose that
$r\le\mu_{1}^{-}\left(\l\right)\le3r$, that $r\le\mu_{2}^{-}(\l) \le3r$,
and that $r\le c\mu$. Then $I\left[\l\right]\ge I\left[\l_{\mu}\right]+r/2$.

(In fact, we may take $c=1/6$.)
\end{lem}

\begin{proof}
Suppose w.l.o.g.\ that $\l(\{2\}) \leq \l(\{1\})$.
We split into two cases: $\l\left(\left\{ 2\right\} \right)\ge\frac{r}{2}$,
and $\l\left(\left\{ 2\right\} \right)\le\frac{r}{2}$.

First suppose that $\l\left(\left\{ 2\right\} \right)\ge\frac{r}{2}$. Note
that, by hypothesis,
$$\tfrac{1}{2} \l(\emptyset) + \tfrac{1}{2} \l(\{2\}) = \mu_1^-(\l) \leq 3r,\quad \tfrac{1}{2} \l(\emptyset) + \tfrac{1}{2} \l(\{1\}) = \mu_2^-(\l) \leq 3r,$$
so
\begin{align*} \l(\{1,2\}) = 4\mu - \l(\emptyset)- \l(\{1\}) -\l(\{2\}) \geq 4\mu - 2\l(\emptyset)- \l(\{1\}) -\l(\{2\}) \geq 4\mu-12r \geq 4r,\end{align*}
provided $c \leq 1/4$. By Lemma \ref{Lem: influences}, we have
\begin{equation}
I\left[\l\right]\ge\tfrac{1}{2}I[\l_{\left\{ 2\right\} }^{\left\{ 2\right\} }]+\tfrac{1}{2}I[\l_{\left\{ 2\right\} }^{\varnothing}]+\mu_{2}^{+}\left(\l\right)-\mu_{2}^{-}\left(\l\right).\label{eq:11}
\end{equation}

Let $\l'$ be the fractional lexicographic family of order 2, such that
\[
\l'\left(\varnothing\right)=0,\ \l'\left(\left\{ 1\right\} \right)=2\mu_{2}^{-}(\l),\ \l'\left(\left\{ 2\right\} \right)=\l\left(\left\{ 2\right\} \right),\ \l'\left(\left\{ 1,2\right\} \right)=\l\left(\left\{ 1,2\right\} \right).
\]
Note that $\mu_2^-(\l) \leq 1/2$ provided $c \leq 1/6$, that $\l'(\{2\}) = \l(\{2\}) \leq \l(\{1\}) \leq 2\mu_2^-(\l) = \l'(\{1\})$, and that $\l'\left(\left\{ 1,2\right\} \right)\ge\l'\left(\left\{ 1\right\} \right)$ provided $c \leq 2/9$.

By Lemma \ref{Lem: influences}, we have
\begin{equation}
I\left[\l'\right]=\tfrac{1}{2}I[\l_{\left\{ 2\right\} }^{\left\{ 2\right\} }]+\tfrac{1}{2}I[\l_{\mu_{2}^-}]+\mu_{2}^{+}\left(\l\right)-\mu_{2}^{-}\left(\l\right).\label{eq:12}
\end{equation}
 By the isoperimetric inequality, (\ref{eq:11}) and (\ref{eq:12}), we have
\begin{equation}
I\left[\l\right]\ge I\left[\l'\right].\label{eq:13}
\end{equation}
Also by Lemma \ref{Lem: influences}, we have
\begin{equation}
I\left[\l'\right]=\tfrac{1}{4}I\left[\l_{\l\left( \left\{ 1,2\right\} \right) }\right]+\tfrac{1}{4}I\left[\l_{\l'\left(\left\{ 1\right\} \right)}\right]+\tfrac{1}{4}I\left[\l_{\l'\left(\left\{ 2\right\} \right)}\right]+\l\left(\left\{ 1,2\right\} \right).\label{eq:13.5}
\end{equation}
 Let $\l''$ be the fractional lexicographic family of order 2, such that
\[
\l''\left(\left\{ \varnothing\right\} \right)=0,\ \l''\left(\left\{ 2\right\} \right)=0,\ \l''\left(\left\{ 1\right\} \right)=\l'\left(\left\{ 1\right\} \right)+\l'\left(\left\{ 2\right\} \right),\ \l''\left(\left\{ 1,2\right\} \right)=\l\left(\left\{ 1,2\right\} \right).
\]
 Note that $\l'(\{1\})+\l'(\{2\}) \leq \l(\{1,2\})$ provided $c \leq 1/6$. By Lemma \ref{Lem: influences}, we have
\begin{equation}
I\left[\l''\right]=\tfrac{1}{4}I\left[\l_{\l\left( \left\{ 1,2\right\} \right) }\right]+\tfrac{1}{4}I\left[\l_{\l'\left(\left\{ 1\right\} \right)+\l'\left(\left\{ 2\right\} \right)}\right]+\l\left(\left\{ 1,2\right\} \right).\label{eq:14}
\end{equation}
The isoperimetric inequality implies that
\begin{align*}
I\left[\l_{\l'\left(\left\{ 2\right\} \right),\l'\left(\left\{ 1\right\} \right)}\right] \ge I\left[\l_{\frac{\l'\left(\left\{ 2\right\} \right)+\l'\left(\left\{ 1\right\} \right)}{2}}\right] & =I\left[\l_{0,\l'\left(\left\{ 1\right\} \right)+\l'\left(\left\{ 2\right\} \right)}\right] \\
& = \tfrac{1}{2} I\left[\l_{\l'\left(\left\{ 1\right\} \right)+\l'\left(\left\{ 2\right\} \right)}\right] + \l'\left(\left\{ 1\right\} \right)+\l'\left(\left\{ 2\right\} \right).
\end{align*}
 Applying Lemma \ref{Lem: influences}, we obtain
\begin{align*}
 \tfrac{1}{2}I\left[\l_{\l'\left(\left\{ 1\right\} \right)}\right]+\tfrac{1}{2}I\left[\l_{\l'\left(\left\{ 2\right\} \right)}\right]+\l'\left(\left\{ 1\right\} \right)-\l'\left(\left\{ 2\right\} \right) & = I\left[\l_{\l'\left(\left\{ 2\right\} \right),\l'\left(\left\{ 1\right\} \right)}\right] \\
 & \ge\tfrac{1}{2}I\left[\l_{\l'\left(\left\{ 1\right\} \right)+\l'\left(\left\{ 2\right\} \right)}\right]+\l'\left(\left\{ 1\right\} \right)+\l'\left(\left\{ 2\right\} \right),
\end{align*}
so rearranging,
\begin{equation} \label{eq:15} \tfrac{1}{2}I\left[\l_{\l'\left(\left\{ 1\right\} \right)}\right]+\tfrac{1}{2}I\left[\l_{\l'\left(\left\{ 2\right\} \right)}\right] \geq \tfrac{1}{2}I\left[\l_{\l'\left(\left\{ 1\right\} \right)+\l'\left(\left\{ 2\right\} \right)}\right] + 2\l'(\{2\}).\end{equation}
Putting everything together, we have
\begin{align*}
I[\l] \geq I[\l'] & \geq \tfrac{1}{4}I\left[\l_{\l\left( \left\{ 1,2\right\} \right) }\right] + \l\left(\left\{ 1,2\right\} \right)+ \tfrac{1}{4}I\left[\l_{\l'\left(\left\{ 1\right\} \right)+\l'\left(\left\{ 2\right\} \right)}\right] + \l'(\{2\})\\
&= I[\l''] + \l(\{2\}) \geq I[\l_{\mu}] + \l(\{2\}) \geq I[\l_{\mu}] + r/2,\end{align*}
as required.

Suppose now that $\l\left(\left\{ 2\right\} \right)\le\frac{r}{2}$. Since $\mu_{1}^{-}\left(\l\right)\ge r$, we have $\l\left(\varnothing\right)\ge\frac{3r}{2}$. Let $\l'$ be the fractional lexicographic family of order 2 such that
\begin{align*}
\l'\left(\varnothing\right) =0,\ \l'\left(\left\{ 2\right\} \right)=\l\left(\left\{ 2\right\} \right),\ \l'\left(\left\{ 1\right\} \right) =\l(\varnothing)+\l\left(\left\{ 1\right\} \right),\ \l'\left(\left\{ 1,2\right\} \right)=\l\left(\left\{ 1,2\right\} \right).
\end{align*}
Note that $\l(\varnothing)+\l(\{1\}) \leq \l(\{1,2\})$, provided $c \leq 2/9$. By Lemma \ref{Lem: influences}, we have
\begin{align*} I[\l] & = \tfrac{1}{4}\left(I[\l_{\l(\varnothing)}]+I[\l_{\l(\{1\})}]+I[\l_{\l(\{2\})}]+I[\l_{\l(\{1,2\})}]\right)\\
&\quad + \l(\{1,2\})-\tfrac{1}{2}\l(\{1\})-\l(\{2\}) + \tfrac{1}{2}\l(\varnothing) +\tfrac{1}{2}|\l(\varnothing)-\l(\{1\})|,
\end{align*}
and
$$I[\l'] = \tfrac{1}{4}\left(I[\l_{\l(\varnothing)+\l(\{1\})}]+I[\l_{\l(\{2\})}]+I[\l_{\l(\{1,2\})}]\right) + \l(\{1,2\}),$$
and therefore
\begin{align} I[\l]-I[\l'] & = \tfrac{1}{4}(I[\l_{\l(\varnothing)}]+I[\l_{\l(\{1\})}] - I[\l_{\l(\varnothing)+\l(\{1\})}])\\
&\quad+ \tfrac{1}{2}(\l(\varnothing) -\l(\{1\}))+\tfrac{1}{2}|\l(\varnothing) - \l(\{1\})|-\l(\{2\})\nonumber \\
& = \tfrac{1}{4}(I[\l_{\l(\varnothing)}]+I[\l_{\l(\{1\})}] - I[\l_{\l(\varnothing)+\l(\{1\})}]) +  \min\{\l(\varnothing) - \l(\{1\}),0\}-\l(\{2\})\label{eq:overall}.
\end{align}
For all $\mu^-,\mu^+ \geq 0$ such that $\mu^- + \mu^+ \leq 1$, we have
$$\tfrac{1}{2} I[\l_{\mu^-}]+ \tfrac{1}{2} I[\l_{\mu^+}] + |\mu^+ - \mu^-| = I[\l_{\mu^+,\mu^-}] \geq \tfrac{1}{2}I[\l_{\mu^+ + \mu^-}] + \mu^+ + \mu^-,$$
using Lemma \ref{Lem: influences} and the isoperimetric inequality, so
$$I[\l_{\mu^-}]+ I[\l_{\mu^+}] -I[\l_{\mu^+ + \mu^-}] \geq 4\min\{\mu^+,\mu^-\}.$$
Applying this with $\mu^- = \l(\varnothing)$ and $\mu^+ = \l(\{1\})$ gives
$$I[\l_{\l(\varnothing)}]+ I[\l_{\l(\{1\})}] -I[\l_{\l(\varnothing)+\l(\{1\})}]) \geq 4\min\{\l(\varnothing),\l(\{1\})\}.$$
Combining this with (\ref{eq:overall}) yields
\begin{align*}
I[\l]-I[\l'] & \geq \min\{\l(\varnothing),\l(\{1\})\} + \min\{\l(\varnothing) -\l(\{1\}),0\}-\l(\{2\})\\
&= \l(\varnothing)- \l(\{2\}) \geq \tfrac{3r}{2}-\tfrac{r}{2}= r.
\end{align*}
The isoperimetric inequality now implies that
\[
I\left[\l\right]\ge I\left[\l'\right]+r \ge I\left[\l_{\mu}\right]+r,
\]
completing the proof.

\begin{comment}
If $\l(\{1\}) \leq \l(\varnothing) - r/4$, then let $\l'$ be the fractional lexicographic family of order 2 such that
\begin{align*}
\l'\left(\varnothing\right) & =\l\left(\left\{ 2\right\} \right),\l'\left(\left\{ 2\right\} \right)=\l\left( \varnothing \right),\\
\l'\left(\left\{ 1\right\} \right) & =\l\left(\left\{ 1\right\} \right),\l'\left(\left\{ 1,2\right\} \right)=\l\left(\left\{ 1,2\right\} \right).
\end{align*}
Then an easy application of Lemma \ref{Lem: influences} gives
$$I[\l] - I[\l'] = \l(\varnothing) - \l(\{1\}) \geq r/4.$$
The isoperimetric inequality now implies
that
\[
I\left[\l\right]\ge I\left[\l'\right]+r/4\ge I\left[\l_{\mu}\right]+r/4,
\]
 as desired. If, on the other hand, $\l(\{1\}) \geq \l(\varnothing) - r/4$, then let $\l''$ be the fractional lexicographic family of order 2 such that
\begin{align*}
\l''\left(\varnothing\right) & =0,\ \l''\left(\left\{ 2\right\} \right)=\l\left(\left\{ 2\right\} \right),\\
\l''\left(\left\{ 1\right\} \right) & =\l(\varnothing)+\l\left(\left\{ 1\right\} \right),\ \l''\left(\left\{ 1,2\right\} \right)=\l\left(\left\{ 1,2\right\} \right).
\end{align*}

Then another easy application of Lemma \ref{Lem: influences} gives
$$I[\l] - I[\l''] \geq \l(\varnothing) - \l(\{2\}) \geq r$$
if $\l(\{1\}) \geq \l(\varnothing)$, and also
$$I[\l] - I[\l''] \geq \l(\varnothing) - \l(\{2\}) \geq r$$
if $\l(\{1\}) \leq \l(\varnothing)$.

The isoperimetric inequality now implies
that
\[
I\left[\l\right]\ge I\left[\l''\right]+r/4\ge I\left[\l_{\mu}\right]+r/4,
\]
 as desired.
\end{comment}
 \end{proof}

%\begin{lem}
%Let $\mu=2^{-j}+r$, where $j \geq 2$ and $0 < r\le2^{-j}$, and
%suppose that $r\le \tfrac{1}{6}\mu$. Let $\l\colon\p\left(\left[2\right]\right)\to\left[0,1\right]$
%be a fractional lexicographic family of order 2 and measure $\mu$, and suppose that
%$r\le\mu_{1}^{-}\left(\l\right)\le3r$ and $r\le\mu_{2}^{-}(\l) \le3r$.
%Then
%\[
%I\left[\l\right]\ge I\left[\l_{\mu}\right]+\tfrac{1}{4}r.
%\]
%
%\end{lem}

\section{\label{sec:Bootstrapping}Two `bootstrapping' lemmas}

In this
section, we prove two `bootstrapping' lemmas which say, roughly speaking, that if $\f$ is `somewhat' close to being contained in a subcube of codimension 1 or 2, then it is `very' close to being contained in that subcube. In what follows, we write $\mu=2^{-j}+r$, where $j \geq 2$ and $0 < r \leq 2^{-j}$, we let $\epsilon>0$,
and we let $\f\subset\pn$ be a family with measure $\mu\left(\f\right)=\mu$,
and with $I\left[\f\right]=I\left[\l_{\mu}\right]+\epsilon$.

Recall that our goal is to prove Proposition \ref{prop:implies Theorem main}. First, we deal with the case where $r$ is `large'. In this case, our aim is to show that $\min\{2\mu_{i}^{-}+\tfrac{1}{2}\epsilon_{i}^{+}, 2\mu_{i}^{+}+\tfrac{1}{2}\epsilon_{i}^{-}\}\le\epsilon$ for some $i\in\left[n\right]$. The following `bootstrapping' lemma says that this inequality holds provided only that $\mu_{i}^{-}(\f) \le r$.
\begin{lem}
\label{lem:bootstrapping r large}Let $0 \leq \mu\le\frac{1}{2}$ and write
$\mu=2^{-j}+r$, where $j \geq 2$ and $r\le2^{-j}$. Let $\f\subset\pn$ be a family
with measure $\mu\left(\f\right)=\mu$, and with $I\left[\f\right]=I\left[\l_{\mu}\right]+\epsilon$.
If $\mu_{i}^{-}: = \mu_i^-\left(\f\right)\le r$ for some $i\in\left[n\right]$,
then $2\mu_{i}^{-}+\frac{1}{2}\epsilon_{i}^{+}\le\epsilon$.\end{lem}
\begin{proof}
Using Lemma \ref{Lem: influences}, the isoperimetric inequality and Lemma \ref{lem:1-lexicographical}, we have
\begin{align*}
I\left[\f\right] & =\tfrac{1}{2}I\left[\f_{\left\{ i\right\} }^{\left\{ i\right\} }\right]+\tfrac{1}{2}I\left[\f_{\left\{ i\right\} }^{\varnothing}\right]+\Inf_{i}\left[\f\right]\\
 & \ge\tfrac{1}{2}I\left[\l_{\mu_{i}^{_{+}}}\right]+\tfrac{1}{2} \epsilon_{i}^{+}+\tfrac{1}{2}I\left[\l_{\mu_{i}^{-}}\right]+\mu_{i}^{+}-\mu_{i}^{-}\\
 & =I\left[\l_{\mu_{i}^{-},\mu_{i}^{+}}\right]+\tfrac{1}{2}\epsilon_{i}^{+}\\
 & \geq I\left[\l_{\mu}\right] + 2\mu_{i}^{-} + \tfrac{1}{2}\epsilon_{i}^{+}.
\end{align*}
Rearranging yields
$$2\mu_{i}^{-} + \tfrac{1}{2}\epsilon_{i}^{+} \leq I\left[\f\right] - I\left[\l_{\mu}\right] = \epsilon,$$
proving the lemma.
\end{proof}
We now prove a bootstrapping lemma suitable for the case where $r$ is `small'. Here, our final goal is to show that there exists a family $\g$ weakly isomorphic to $\f$ such that either $c\mu_{1}^{-}(\g) +\frac{1}{2}\epsilon_{1}^{+}(\g) \leq \epsilon$, or else $c\mu\left(\g\backslash\s_{\left\{1,2\right\} }\right)+\frac{1}{4}\epsilon_{1,2}^{++}(\g)\le\epsilon$. We show that one of these inequalities holds provided $\mu_{1}^{-}(\g) \le\mu_{2}^-(\g) \le c\mu$, if $c$ is
a sufficiently small positive constant.
\begin{lem}
\label{lem:general bootstrapping} Let $\epsilon > 0$, let $0 < \mu \leq \tfrac{1}{2}$, and write $\mu=2^{-j}+r$, where $j \geq 2$ and $0 < r\le2^{-j}$.
Let $\f\subset\pn$ be a family with measure $\mu\left(\f\right)=\mu$,
and with $I\left[\f\right]=I\left[\l_{\mu}\right]+\epsilon$. Suppose that $\mu_{1}^{-}\left(\f\right)\le\mu_{2}^{-}\left(\f\right)\le \tfrac{1}{6} \mu$.
Then either $\tfrac{2}{3}\mu_{2}^{-}(\f)+\frac{1}{2}\epsilon_{2}^{+}(\f) \le\epsilon$, or $2\mu_{1}^{-}(\f) +\frac{1}{2}\epsilon_{1}^{+}(\f) \le\epsilon$, or $\tfrac{1}{6}\mu\left(\f\backslash\s_{\left\{ 1,2\right\} }\right)+\frac{1}{4}\epsilon_{1,2}^{++}(\f) \le\epsilon$.\end{lem}
\begin{proof}
The case where $\mu_{1}^{-}\le r$
is covered by Lemma \ref{lem:bootstrapping r large}, and the case
where $\mu_{2}^{-}\ge3r$ can be covered similarly by using the second
part of Lemma \ref{lem:1-lexicographical} instead of its first part. So we may assume that $r\le\mu_{1}^{-}\le\mu_{2}^{-}\le3r$.

Let $\l$ be the fractional lexicographic family of order 2, with
\[
\l\left(\left\{ 1,2\right\} \right)=\mu_{1,2}^{++},\ \l\left(\left\{ 1\right\} \right)=\mu_{1,2}^{+-},\ \l\left(\left\{ 2\right\} \right)=\mu_{1,2}^{-+},\ \l\left(\left\{ \varnothing\right\} \right)=\mu_{1,2}^{--}.
\]
Using Lemma \ref{Lem: influences}, the isoperimetric inequality and Lemma \ref{lem:2-lex}, we obtain
\[
I\left[\f\right]\ge I\left[\l\right]+\tfrac{1}{4}\epsilon_{1,2}^{++} \geq I\left[\l_{\mu}\right]+\tfrac{1}{4}\epsilon_{1,2}^{++}+ \tfrac{1}{2}r,
\]
and therefore
$$\epsilon \geq \tfrac{1}{4}\epsilon_{1,2}^{++}+ \tfrac{1}{2}r \geq \tfrac{1}{4}\epsilon_{1,2}^{++} + \tfrac{1}{12}(\mu_1^- + \mu_2^-) \geq  \tfrac{1}{4}\epsilon_{1,2}^{++} + \tfrac{1}{6}\mu(\f \setminus \s_{\{1,2\}}),$$
proving the lemma.
\end{proof}

\section{$\f$ is essentially contained in a codimension-1 subcube ($\mu$ large)}
\label{sec:large}

In this section we essentially complete the proof of Proposition \ref{prop:implies Theorem main} in the case where $\mu\left(\f\right)=\frac{1}{4}+r$, for $r\le\frac{1}{4}$ bounded away from $0$ (i.e. $r\ge c_{1}$ for some absolute constant $c_{1}>0$). In this case, by Lemma~\ref{lem:bootstrapping r large} it will suffice to prove the following.
\begin{prop}
\label{prop:mu large r large} For each $c_{1}>0$ there exists $c_{2} = c_2\left(c_{1}\right)>0$ such that the following holds.
If $\f\subset \pn$ is a family with $\frac{1}{4}+c_{1}\le\mu:=\mu\left(\f\right)\le\frac{1}{2}$
and $I\left[\f\right]\le I\left[\l_{\mu}\right]+c_{2}$, then
there exists a family $\g$ weakly isomorphic to $\f$ such that $\mu_{1}^{-}\left(\g\right)\le c_{1}$.
\end{prop}
Equivalently, the conclusion of Proposition \ref{prop:mu large r large}
can be restated by saying that there exists a coordinate $i\in\left[n\right]$
such that $\min\left\{ \mu_{i}^{-}(\f),\mu_{i}^{+}(\f)\right\} \le c_{1}$.

Throughout this section, $\f\subset\pn$ will be a family with $\frac{1}{4}+c_{1}\le\mu:=\mu\left(\f\right)\le\frac{1}{2}$,
and with $I\left[\f\right]\le I\left[\l_{\mu}\right]+c_{2}$, where
$c_{2}$ will be sufficiently small in terms of $c_{1}$. We assume without loss of generality that $\mu_{i}^{-}\le\mu_{i}^{+}$
for each $i\in\left[n\right]$, and that
\[
\Inf_{n}[\f]\le\Inf_{n-1}[\f]\le\cdots\le\Inf_{1}[\f].
\]
 We also assume that $c_{1}=2^{-k}$ for some $k\in\mathbb{N}$.
 \begin{comment}
 We use the notation $\mathcal{I}_{i}\left(\f\right):=\left\{ A\,:\, A\in\f,\text{ and }A\Delta\left\{ i\right\} \notin\f\right\} $.
\end{comment}

\medskip

The proof of Proposition \ref{prop:mu large r large} consists of the following six steps, similarly to in \cite{Keller Lifshitz}.
\begin{enumerate}
\item We show that there is a `gap' between `good' families that satisfy the
proposition, and `bad' families which would furnish a counterexample to it. More precisely,
we show that if $\mu_{i}^{-}\left(\f\right)>c_{1}$ for each $i\in\left[n\right]$
and if $\f_{2}\subset\pn$ is a family with $\mu\left(\f_{2}\right)=\mu\left(\f\right)$
and with $I\left[\f_{2}\right]\le I\left[\f\right]$, which satisfies
$\mu_{i}^{-}\left(\f_{2}\right)\le c_{1}$ for some $i\in\left[n\right]$,
then $\mu\left(\f \Delta\f_{2}\right) > \frac{c_{1}}{2}$.
\item We reduce the proposition to the case where $\f$ is increasing.
\item We prove the proposition in the case where
$\f$ depends on a constant $O_{c_{2}}\left(1\right)$ number of coordinates.
\item In the other case, where $n$ is large, we show that the `$n$-stable'
family
\[
\tilde{\f}:=\s_{n,n-1}\left(\s_{n,n-2}\cdots\left(\s_{n,1}\left(\f\right)\right)\right)
\]
 satisfies $\mu(\tilde{\f})=\mu(\f)$, $I[\tilde{\f}]\le I[\f]$,
and $\mu(\tilde{\f}\Delta\f)\le\frac{c_{1}}{2}$. This
reduces us to the case where $\f$ is an increasing, $n$-stable family.
\item In the case where $\f$ is $n$-stable and $|\mathcal{I}_n(\f)| \geq 2$ (say $A\ne B\in\mathcal{I}_{n}\left(\f\right)$), we show that if both $\f_{1}:=(\f\backslash\left\{ B\right\}) \cup\left\{ A\backslash\left\{ n\right\} \right\}$
and $\f_{2}:=(\f\backslash\left\{ A\right\}) \cup \{B\backslash\left\{ n\right\}\}$
are good, then $\f$ is also good. (Note that $\left|\mathcal{I}_{n}\left(\f_{1}\right)\right|<|\mathcal{I}_{n}\left(\f\right)|$,
and that $\left|\mathcal{I}_{n}\left(\f_{2}\right)\right|<\left|\mathcal{I}_{n}\left(\f\right)\right|$.)
\item Step (5) reduces us to the case where $\left|\mathcal{I}_{n}\left(\f\right)\right|\le1$, i.e. the family $\f$ is very evenly balanced in direction $n$; we can then complete the proof by induction on $n$.
\end{enumerate}

\subsection{Gap between good families and bad families}

If $0 \leq s \leq 1$ and $\f,\g \subset \pn$, we say $\g$ is an \emph{$s$-small modification}
of $\f$ if $\mu\left(\g\right)=\mu\left(\f\right)$, $I\left[\g\right]\le I\left[\f\right]$
and $\mu\left(\f\Delta\g\right)\le s$. If $\f \subset \pn$ and $c_1,c_2>0$, we say that $\f$
is \emph{bad} (with respect to $(c_1,c_2)$) if it is a counterexample to Proposition \ref{prop:mu large r large}, and {\em good} (with respect to $(c_1,c_2)$) otherwise.
\begin{lem}
\label{lem:gap} Let $\f\subset \pn$ such that $I\left[\f\right]\le I\left[\l_{\mu}\right]+c_{2}$. Let $\g$ be a $\frac{c_1}{2}$-small modification of
$\f$, and suppose that $c_2 \leq c_1$. If $\g$ is good, then so is $\f$.\end{lem}
\begin{proof}
Suppose for a contradiction that $\f$ is bad and $\g$ is good. Then by assumption,
\[
I\left[\g\right]\le I\left[\f\right]\le I\left[\l_{\mu}\right]+c_{2}.
\]
 Since $\g$ is good, we either have $\mu_{i}^{-}\left(\g\right)\le c_{1}$
or $\mu_{i}^{+}\left(\g\right)\le c_{1}$ for some $i\in\left[n\right]$.
By Lemma \ref{lem:bootstrapping r large}, this implies that we either
have
\[
\mu_{i}^{-}\left(\g\right)\le \tfrac{1}{2} c_{2} \le \tfrac{1}{2} c_{1}
\]
 or
\[
\mu_{i}^{+}\left(\g\right)\le\tfrac{1}{2} c_{2}\le\tfrac{1}{2} c_{1},
\]
 since $c_{2}\le c_{1}$. Since $\mu\left(\g\Delta\f\right)\leq \frac{c_{1}}{2}$, we have $\mu(\f \setminus \g) \leq \frac{c_1}{4}$, and therefore $\mu_{i}^{-}\left(\f\right)\le \mu_i^-(\g) + 2\mu(\f \setminus \g) \leq \frac{c_{1}}{2}+\frac{c_1}{2} = c_1$ if $\mu_{i}^{-}\left(\g\right)\le\frac{c_{1}}{2}$,
and $\mu_{i}^{+}\left(\f\right)\le \mu_i^+(\g) + 2\mu(\f \setminus \g) \leq \frac{c_{1}}{2}+\frac{c_1}{2} \leq c_{1}$
if $\mu_{i}^{+}(\g) \le\frac{c_{1}}{2}$. Hence, $\f$ is
good, a contradiction.
\end{proof}

\subsection{Reduction to the case where $\f$ is increasing}

Here we show that one can transform $\f$ into an increasing family
by a series of $\frac{c_{1}}{2}$-small modifications. (For brevity, if $i \in [n]$ we henceforth write $\s_{\varnothing i}$ for $\s_{\varnothing \{i\}}$.) It suffices to prove the following lemma.
\begin{lem}
\label{lem:small modification}
Let $\f\subset\pn$ with $\mu(\f) = \mu$, $\mu_i^-(\f) \leq \mu_i^+(\f)$ for all $i \in [n]$, and $I\left[\f\right]\le I\left[\l_{\mu}\right]+c_{1}$. Then for each $i \in [n]$,
$\s_{\varnothing i}\left(\f\right)$ is a $\frac{c_{1}}{2}$-small
modification of $\f$.\end{lem}
\begin{proof}
By Lemma \ref{lem:mon dec inf}, we have $I[\s_{\varnothing i}\left(\f\right)] \leq I[\f]$. By the isoperimetric inequality, we have
\[
\left(\mu_{i}^{+}-\mu_{i}^{-}\right)+\tfrac{1}{2}I\left[\l_{\mu_{i}^{-}}\right]+\tfrac{1}{2}I\left[\l_{\mu_{i}^{+}}\right]=I\left[\l_{\mu_{i}^{-},\mu_{i}^{+}}\right]\ge I\left[\l_{\mu}\right].
\]
 By Lemma \ref{Lem: influences}, and by the isoperimetric inequality
applied to $\f_{\left\{ i\right\} }^{\left\{ i\right\}}$ and to $\f_{\left\{ i\right\} }^{\varnothing}$,
we have
\begin{equation}
\Inf_{i}\left[\f\right]+\tfrac{1}{2}I\left[\l_{\mu_{i}^{-}}\right]+\tfrac{1}{2}I\left[\l_{\mu_{i}^{+}}\right]\le I\left[\f\right]\le I\left[\l_{\mu}\right]+c_{1}.\label{eq:-1}
\end{equation}
 These imply that
\begin{equation}
\Inf_{i}\left[\f\right]-\left(\mu_{i}^{+}-\mu_{i}^{-}\right)\le c_{1}.\label{eq:-2}
\end{equation}
Now note that
\begin{equation}
\Inf_{i}\left[\f\right]=\mu_{i}^{+}\left(\s_{\varnothing i}\left(\f\right)\right)-\mu_{i}^{-}\left(\s_{\varnothing i}\left(\f\right)\right).\label{eq:-3}
\end{equation}
 Combining (\ref{eq:-2}) and (\ref{eq:-3}), we obtain
\[
\mu\left(\f\Delta\s_{\varnothing i}\left(\f\right)\right)=\tfrac{1}{2}\left(\mu_{i}^{+}\left(\s_{\varnothing i}\left(\f\right)\right)-\mu_{i}^{+}\left(\f\right)\right)+\tfrac{1}{2}\left(\mu_{i}^{-}\left(\f\right)-\mu_{i}^{-}\left(\s_{\varnothing i}\left(\f\right)\right)\right)\le\tfrac{c_{1}}{2},
\]
 as required.
\end{proof}

The next corollary follows immediately from Lemmas \ref{lem:gap} and \ref{lem:small modification}.

\begin{cor}
\label{cor:mon}
Let $\f \subset \pn$ with $\mu(\f) = \mu$, $\mu_i^-(\f) \leq \mu_i^+(\f)$ for all $i \in [n]$ and $I[\f] \leq I[\l_\mu]+c_2$, let $c_2 \leq c_1$, and let
$$\g:=(\s_{\varnothing n}\circ\s_{\varnothing n-1}\circ\cdots\circ\s_{\varnothing1})\left(\f\right).$$
If the increasing family $\g$ is good, then so is $\f$.
\end{cor}

From now on we assume that $\f$ is increasing. 

\subsection{Proof in the case where $n$ is small}

We now show that Proposition \ref{prop:mu large r large} holds in
the case where $n$ is small. In fact, crudely, we have the following.
\begin{lem}
\label{lem:n-small}Suppose that $n \leq n_0$ and $\f \subset \pn$ with $\mu(\f) = \mu$ and $I[\f] \leq I[\l_\mu]+c_2$. Then
$\f$ is weakly isomorphic to $\l_{\mu}$, provided $c_{2} < 2^{-(n_0-1)}$.\end{lem}
\begin{proof}
If $c_{2}< 2^{-(n-1)}$, then we must have $I\left[\f\right]=I\left[\l_{\mu}\right]$ (note
that the influence of any family depending on $n$ variables is of
the form $\frac{i}{2^{n-1}}$). The lemma now follows from the uniqueness
part of the isoperimetric inequality.
\end{proof}

\subsection{Reduction to the case where $\f$ is $n$-stable}

We say that $\f$ is \emph{n-stable }if $\s_{n,i}\left(\f\right)=\f$
for each $i \in [n-1]$, and if $A\cup\left\{ n\right\} \in\f$
whenever $A\in\f$. (As usual, we write $\s_{i,j}$ for $\s_{\{i\}\{j\}}$, for brevity.) Here, we show that if $\f\subset\pn$ is bad
and $n$ is large then there exists a $\frac{c_1}{2}$-small modification
of $\f$ that is $n$-stable. We need the following well-known lemma.
\begin{lem}
\label{lem:shifting decreases influence} Let $i \neq j \in [n]$. Then $\mu\left(\s_{i,j}\left(\f\right)\right)=\mu\left(\f\right)$ and $I\left[\s_{i,j}\left(\f\right)\right]\le I\left[\f\right]$.
\end{lem}

We remark that the operator $\s_{i,j}$ preserves monotonicity, for each $i \neq j$. We also need the following crude upper bound on the total influence of lexicographically ordered families.
\begin{claim}
\label{claim:crude}
$I\left[\l_{\mu}\right]\le2$ for each $\mu\in\left[0,1\right]$.
\end{claim}
\begin{proof}
We prove the claim by induction on $n$. If $n=1$ then in fact $I\left[\l_{\mu}\right]\le1$. We may assume that $\mu\le\frac{1}{2}$,
since $I\left[\l_{1-\mu}\right]=I\left[\l_{\mu}\right]$. Hence, the
induction hypothesis implies that
\[
I\left[\l_{\mu}\right]=I\left[\l_{0,2\mu}\right]=2\mu+\tfrac{1}{2}I\left[\l_{2\mu}\right]\le2.
\]
\end{proof}

The following lemma reduces Proposition \ref{prop:mu large r large} to the case where $n$ is stable.
\begin{lem}
\label{lem:stabilizing}Let $\f \subset \pn$ be an increasing family with $\mu(\f) = \mu$ and $I[\f] \leq I[\l_\mu]+c_2$, suppose that $\Inf_n[\f] = \min_{i \in [n]} \Inf_i[\f]$, and let
$$\tilde{\f}:=\s_{n,n-1}\left(\cdots\s_{n,2}\left(\s_{n,1}\left(\f\right)\right)\right).$$
If $c_2 < \min\{2^{-(6/c_1-1)},1,c_1\}$, and the $n$-stable family $\tilde{\f}$ is good, then so is $\f$.\end{lem}
\begin{proof}
By Lemma \ref{lem:gap}, it suffices to show that $\tilde{\f}$ is a $\frac{c_1}{2}$-small modification
of $\f$. For this, by Lemma \ref{lem:shifting decreases influence}, it suffices to prove that $\mu\left(\tilde{\f}\Delta\f\right)\le\frac{c_{1}}{2}$.
The key observation is that
\begin{equation}
\mu(\tilde{\f}\Delta\f)=2\mu(\f\backslash\tilde{\f})\le2\mu(\mathcal{I}_{n}(\f)).\label{eq:111}
\end{equation}
Indeed,
$$A \mapsto A \cup \{i_{A}\},$$
where $i_A := \min\{i:\ (A \cup \{i\}) \setminus \{n\} \notin \f\}\}$, is an injection from $\f\backslash\tilde{\f}$ to $\mathcal{I}_n(\f)$.
Now note that
\begin{equation}
\mu(\mathcal{I}_{n}(\f))=\frac{\Inf_{n}\left(\f\right)}{2}\le\frac{I\left(\f\right)}{2n}\le\frac{I\left[\l_{\mu}\right]+c_{2}}{2n}.\label{eq:112}
\end{equation}
To complete the proof of the lemma, note that (\ref{eq:111}) and (\ref{eq:112}) imply
that
\[
\mu\left(\tilde{\f}\Delta\f\right)\le\frac{I\left[\l_{\mu}\right]+c_{2}}{n} < \frac{3}{n},
\]
provided $c_2 < 1$, using Claim \ref{claim:crude}. Suppose that $\f$ is bad. By Lemma \ref{lem:n-small}, we may assume that $n\geq 6/c_{1}$, provided $c_2 < 2^{-(6/c_1-1)}$,
and therefore $\mu(\tilde{\f}\Delta\f)\le\tfrac{1}{2}c_1$. Hence, $\tilde{\f}$ is a $\frac{c_1}{2}$-small modification of $\f$, and we are done by Lemma \ref{lem:gap}, provided $c_2 \leq c_1$.
\end{proof}
From now on we assume also that $\f$ is $n$-stable.

\subsection{The case where $\left|\mathcal{I}_{n}\right|\ge2$}

Here we show that if $\left|\mathcal{I}_{n}\right|\ge2$, then there
exists a $\frac{1}{2^{n-1}}$-small modification of $\f$ with smaller
$\left|\mathcal{I}_{n}\right|$. This will allow us to reduce to the case where
$\left|\mathcal{I}_{n}\right|\le1$.
\begin{lem}
\label{lem:stable families.}Let $A,B\in\mathcal{I}_{n}\left(\f\right)$ such that $A \neq B$,
and let $\f_{1}=(\f\backslash\left\{ A\right\}) \cup \{B\backslash\left\{ n\right\} \}$,
$\f_{2}=(\f\backslash \{B\})\cup \{A\backslash\left\{ n\right\}\}$.
Then either $I\left[\f_{1}\right]<I\left[\f\right]$ or $I\left[\f_{2}\right]<I\left[\f\right]$.\end{lem}
\begin{proof}
Suppose w.l.o.g.\ that $\left|A\right|\ge\left|B\right|$; we will show
that $I\left[\f_{1}\right] < I\left[\f\right]$. Since $\f$ is $n$-stable, we have
$A\backslash\left\{ i\right\} \notin\f$ for each $i\in A$ (note
that $\left((A\backslash\left\{ i\right\} \right)\backslash\left\{ n\right\}) \cup\left\{ i\right\} =A\backslash\left\{ n\right\} \notin\f$).
This implies that
$$|\partial (\f \setminus \{A\})| - |\partial \f| \leq n-2\left|A\right|.$$
Since $(B\backslash\left\{ n\right\}) \cup\left\{ i\right\} \in\f$
for each $i\notin B\backslash\left\{ n\right\}$, we have
$$|\partial ((\f \setminus \{A\}) \cup \{B \setminus \{n\}\})| - |\partial (\f \setminus \{A\})| \leq 2\left|B\right|-n-2.$$
This implies that $|\partial\f_{1}|\le|\partial \f|+\left(n-2\left|A\right|\right)+\left(2\left|B\right|-n-2\right)\leq |\partial \f|-2$.
\end{proof}

\subsection{Proof of Proposition \ref{prop:mu large r large}}

%\begin{proof}[Proof of Proposition \ref{prop:mu large r large}]
We prove the proposition by induction on $n$. Let $c_1 = 2^{-k}$, where $k \in \mathbb{N}$. Assume that
$$c_2 < \min\{2^{-(6/c_1-1)},1,c_1\}.$$
The case $n \leq k+1$ follows
from Lemma \ref{lem:n-small}. Let $n \geq k+2$, and let $\f$ be as in the hypothesis of the proposition. Suppose for a contradiction that $\f$
is bad. By Corollary \ref{cor:mon}, we may assume that $\f$ is increasing; by Lemma \ref{lem:stabilizing}, we may assume that $\f$
is $n$-stable, and by Lemmas \ref{lem:gap}, \ref{lem:n-small},
and \ref{lem:stable families.} we may assume that $\left|\mathcal{I}_{n}\left(\f\right)\right|\le1$.
If $\left|\mathcal{I}_{n}\left(\f\right)\right|=0$, then $\f$ does not depend on the $n$th coordinate and the proposition holds by the
induction hypothesis. Suppose that $\left|\mathcal{I}_{n}\left(\f\right)\right|=1$.
Then, by Lemma \ref{Lem: influences}, we have
\begin{equation}
I\left[\f\right]=\tfrac{1}{2}I\left[\f_{\left\{ n\right\} }^{\left\{ n\right\} }\right]+\tfrac{1}{2}I\left[\f_{\left\{ n\right\} }^{\varnothing}\right]+\Inf_{n}\left[\f\right]=\tfrac{1}{2}I\left[\f_{\left\{ n\right\} }^{\left\{ n\right\} }\right]+\tfrac{1}{2}I\left[\f_{\left\{ n\right\} }^{\varnothing}\right]+\frac{1}{2^{n-1}}.\label{eq:1111}
\end{equation}
Since $\left|\f\right|$ is odd, and since in the lexicographic ordering, sets containing $n$ alternate with sets not containing $n$, we have $\l_{\mu}=(\l_{\mu_{n}^{-},\mu_{n}^{+}})_{\textrm{ass}}$, and therefore
\begin{equation}
I\left[\l_{\mu}\right]=\tfrac{1}{2}I\left[\l_{\mu_{n}^{-}\left(\f\right)}\right]+\tfrac{1}{2}I\left[\l_{\mu_{n}^{+}\left(\f\right)}\right]+\frac{1}{2^{n-1}}.\label{eq:1112}
\end{equation}

By (\ref{eq:1111}) and (\ref{eq:1112}), we either have $I[\f_{\left\{ n\right\} }^{\left\{ n\right\} }]\le I[\l_{\mu_{n}^{+}}]+c_{2}$
or $I[\f_{\left\{ n\right\} }^{\varnothing}]\le I[\l_{\mu_{n}^{-}}]+c_{2}$.
Recall that we are assuming $c_{1} = 2^{-k}$ for some $k \in \mathbb{N}$, and that $n \geq k+2$. Since $1/4+2^{-k} \leq \mu =\mu_{n}^{-}+2^{-n}$, $2^{n-1}\mu_n^- \in \mathbb{Z}$ and $n > k$, we have $\mu_{n}^{-}\geq 1/4+2^{-k}$.
Moreover, since $\mu \leq 1/2$, $\mu_n^+ = \mu + 2^{-n}$ and $2^{n-1}\mu_n^+ \in \mathbb{Z}$, we must have $\mu \leq 1/2-2^{-n}$, and therefore $\mu_n^+ \leq 1/2$. Hence,
$$\tfrac{1}{4} + c_1 \leq \mu_n^- < \mu_n^+ \leq \tfrac{1}{2}.$$
Therefore, we may apply the induction hypothesis to one of $\f_{\left\{ n\right\} }^{\varnothing}$ and $\f_{\left\{ n\right\} }^{\{n\}}$. Suppose first that $\mu_{i}^{-}(\f_{\left\{ n\right\} }^{\varnothing})\le c_{1}$.
Then, by Lemma \ref{lem:bootstrapping r large}, we have $\mu_{i}^{-}(\f_{\left\{ n\right\} }^{\varnothing})\le\frac{c_{2}}{2}$.
This implies that
$$\mu_{i}^{-}\left(\f\right)\le\mu_{i}^{-}\left(\f_{\left\{ n\right\} }^{\varnothing}\right)+2^{-n} \le \tfrac{1}{2} c_{2} +2^{-n} \le c_{1} = 2^{-k}$$
since $c_{2}\leq c_1$ and $n > k$. Hence, $\f$ is good, as desired.
The case where $I[\f_{\left\{ n\right\} }^{\left\{ n\right\} }]\le I[\l_{\mu_{n}^{+}\left(\f\right)}]+c_{2}$
is similar.

%\end{proof}

\section{$\f$ is essentially contained in a codimension-2 subcube ($\mu$ small)}
\label{sec:small}

In this section we essentially complete the proof of Proposition \ref{prop:implies Theorem main} in the case where $\mu\left(\f\right)$ is `small'. Specifically, we prove the following.
\begin{prop}
\label{prop:mu small}For each $c>0$, there exists $d = d\left(c\right)>0$ such that the following holds. Suppose $\f \subset \pn$ with $\mu:=\mu\left(\f\right)\le\frac{1}{4}+d$
and $I\left[\f\right]\le I\left[\l_{\mu}\right]+d\mu$. Then
there exists a family
$\g$ weakly isomorphic to $\f$, such that $\mu_{1}^{-}\left(\g\right)\le\mu_{2}^{-}\left(\g\right)\le c\mu$.
\end{prop}

We start by reducing to the case where $\f$ is increasing.
\begin{lem}
\label{lem:mono-suff}
If Proposition \ref{prop:mu small} holds for all increasing families
$\f$, then it holds for all families $\f.$ \end{lem}
\begin{proof}
Given $c>0$, let $d' = d'(c)>0$ such that each increasing family $\g$ with $\mu(\g) \leq 1/4+d'$, satisfying $I\left[\g\right]\le I\left[\l_{\mu(\g)}\right]+d'\mu(\g)$
also satisfies $\mu_{i}^{-}\left(\g\right)\le\mu_{j}^{-}\left(\g\right) \leq \frac{c\mu(\g)}{2}$ for some $i \neq j \in [n]$.
Let $d=\min\left\{c,d'\right\}$. Let $\f \subset \pn$ be some
family satisfying $\mu:=\mu(\f) \leq 1/4+d$ and $I\left[\f\right]\le I\left[\l_{\mu}\right]+d\mu$.
We may assume without loss of generality that $\mu_{i}^{-}(\f) \le\mu_{i}^{+}(\f)$
for each $i\in\left[n\right]$. By Lemma \ref{lem:small modification}, we have
$\mu\left(\s_{\varnothing1}\left(\f\right)\Delta\f\right)\le\frac{d\mu}{2}$, and therefore
\begin{equation}
\mu\left(\f_{\left\{ 1\right\} }^{\varnothing}\right)\le \mu\left(\left(\s_{\varnothing1}\left(\f\right)\right)_{\left\{ 1\right\} }^{\varnothing}\right) +\tfrac{1}{2} d \mu.\label{eq:7.2}
\end{equation}
 Note also that $\mu(\s_{\varnothing1}\left(\f\right)_{\left\{ i\right\} }^{\varnothing})=\mu(\f_{\left\{ i\right\} }^{\varnothing})$
for any $i\ge2$. Now let
$$\g:=\s_{\varnothing n}\left(\s_{\varnothing\left(n-1\right)}\cdots\left(\s_{\varnothing1}\left(\f\right)\right)\right);$$
clearly, $\g$ is increasing. As in (\ref{eq:7.2}), we have
\[
\mu\left(\f_{\left\{ i\right\} }^{\varnothing}\right)\le\mu\left(\g_{\left\{ i\right\} }^{\varnothing}\right)+\tfrac{1}{2} d\mu \quad \forall i\in [n].
\]

By Lemma \ref{lem:mon dec inf}, we have $I\left[\g\right]\le I\left[\f\right]\le I\left[\l_{\mu}\right]+d\mu$.
Since $d \leq d'$, there exist two coordinates $i \neq j \in [n]$ such that
$\mu_{i}^{-}\left(\g\right)\le\mu_{j}^{-}\left(\g\right) \leq \frac{c\mu}{2}$.
Hence, $\mu_{i}^{-}\left(\f\right)\le\mu_{i}^{-}\left(\g\right)+\frac{d\mu}{2} \leq c\mu$,
and similarly, $\mu_{j}^{-}\left(\f\right)\le\mu_{j}^{-}\left(\g\right)+\frac{c\mu}{2}\leq c\mu$, as required.
\end{proof}

The key lemma for the proof of Proposition \ref{prop:mu small} is the following.
\begin{lem}
\label{lem:Shifting}Let $\f\subset\pn$ be increasing, with
$\mu\left(\f\right)\le\frac{1}{2}$. Let
\[
\f_{1}=\s_{n,1}\s_{n-1,1}\circ\cdots\circ\s_{2,1}\left(\f\right),
\]
\[
\f_{2}=\s_{\left\{ n,n-1\right\} 1}\circ\cdots\circ\s_{\left\{ 3,2\right\} 1}\left(\f_{1}\right),
\]

\[
\vdots
\]

\[
\f_{n}=\s_{\left\{ n,n-1,\ldots,2\right\} 1}\left(\f_{n-1}\right).
\]
Then
\begin{itemize}
\item[(i)] $\f_{n}$ is contained in the dictatorship $\d_{1}$,
\item[(ii)] $I\left[\f_{n}\right]\le I\left[\f\right]$, and
\item[(iii)] $\mu_{i}^{-}\left(\f_{n}\right)\ge\mu_{i}^{-}\left(\f\right)$
for any $i > 1$.
\end{itemize}
\end{lem}
\begin{proof}
To prove (i), first note that $\f_n$ is increasing, using the fact that $\s_{S1}(\g)$ is increasing whenever $\g$ is increasing and $\s_{S'1}(\g)=\g$ for all $S' \subset S$ with $|S'| = |S|-1$. Suppose for a contradiction that $\f_n \nsubseteq \d_{1}$; then there exists $S \subset \{2,\ldots,n\}$ such that $S \in \f_n$, so by the monotonicity of $\f_n$, we have $\{2,3,\ldots,n\} \in \f_n$. But then, by construction of $\f_n$, we have $\d_1 \subset \f_n$, and so $\d_1 \cup \{\{2,3,\ldots,n\}\} \subset \f_n$, contradicting the fact that $\mu(\f_n) = \mu(\f) \leq 1/2$.

Statement (ii) follows by repeated application of Lemma \ref{lem:shifting dec inf}, and (iii) is clear.
\end{proof}

The idea of the proof of Proposition \ref{prop:mu small} is as follows. Let $\f \subset \pn$ be an increasing family as in the hypothesis of the proposition; assume w.l.o.g.\ that $\mu_1^-(\f) = \min_i(\mu_i^-(\f))$. Let $\f_{n}$ be the family from Lemma \ref{lem:Shifting}. By Lemma
\ref{lem:Shifting}, we have $\mu_{i}^{-}((\f_{n})_{\left\{ 1\right\} }^{\{1\}})=2\mu_{i}^{-}(\f_{n})\ge 2\mu_{i}^{-}(\f)$ for all $i >1$, and $\mu((\f_{n})_{\left\{ 1\right\} }^{\{1\}}) = 2\mu(\f)$. This allows us perform an inductive argument, doubling the measure of the family at each step, and thus reducing to the case of measure somewhat larger than $1/4$ (encapsulated in the following lemma, which enables us to do the base case of the induction).
\begin{lem}
\label{lem:base for induction}For each $c>0$, there exist $d_1=d_1\left(c\right)>0,\ d_2 = d_2(c)>0$ such that the following holds. Suppose
that $\f \subset \pn$ is increasing with $\frac{1}{4}+d_1\le\mu\left(\f\right)\le\frac{1}{2}+2d_1$,
and that
\[
I\left[\f\right]\le I\left[\l_{\mu\left(\f\right)}\right]+d_2\mu\left(\f\right).
\]
 Then $\mu_{i}^{-}\left(\f\right) \le c\mu(\f)$ for some $i\in\left[n\right]$.\end{lem}
\begin{proof}
If $\mu\left(\f\right)\le\frac{1}{2}$, the lemma follows from applying Proposition \ref{prop:mu large r large} to $\f$. In the case where $\mu\left(\f\right)\ge\frac{1}{2}$,
it follows by applying Proposition \ref{prop:mu large r large} to $\f^c$.
\end{proof}
We now prove Proposition \ref{prop:mu small}.
\begin{proof}[Proof of Proposition \ref{prop:mu small}]
By Lemma \ref{lem:mono-suff}, it suffices to prove the proposition for increasing families. Let $c>0$, and let $\f \subset \pn$ be an increasing family as in the statement of the proposition, where $d=\min\{d_1(c),d_2(c)\}$ and $d_1(c),d_2(c)$ are as in Lemma \ref{lem:base for induction}. Write $\mu:=\mu(\f)$. Write $\mu=2^{-j}\cdot\mu_{0}$,
where $\frac{1}{4}+d_1\le\mu_{0}\le\frac{1}{2}+2d_1$ and $j\in \mathbb{N}$. We prove by induction on $j$ that Proposition \ref{prop:mu small} holds (for increasing families) with the above choice of $d$. Let $j \geq 1$. Suppose w.l.o.g.\ that $\f$ satisifes
$$\mu_{1}^{-}(\f) \le\mu_{2}^{-}(\f) \le\cdots\le\mu_{n}^{-}(\f).$$
Let $\f_{n}$ be as in Lemma
\ref{lem:Shifting}. Let $\f'=\left(\f_{n}\right)_{\left\{ 1\right\} }^{\left\{ 1\right\} }$; then $\mu(\f') = 2\mu$.
If $j=1$, then $\mu(\f') = \mu_0 \in [1/4+d_1,1/2+2d_1]$, and so it follows from Lemma \ref{lem:base for induction} that
\[
\min_{i>1}\mu_{i}^{-}\left(\f'\right)\le c\mu\left(\f'\right)=2c\mu.
\]
Suppose that $\min_{i>1}\mu_{i}^{-}\left(\f'\right)=\mu_{m}^{-}\left(\f'\right)$. Then
\[
\mu_{1}^{-}\left(\f\right)\le \mu_2^-(\f) \leq \mu_{m}^{-}\left(\f\right)\le\mu_{m}^{-}\left(\f_{n}\right)=\tfrac{1}{2}\mu_{m}^{-}\left(\f'\right)\leq c\mu,
\]
completing the base case of the induction. Now let $j \geq 2$, and assume the desired statement holds when $j$ is replaced by $j-1$. Applying the induction hypothesis to $\f'$ yields
\[
\min_{i>1}\mu_{i}^{-}\left(\f'\right)\le c\mu\left(\f'\right)=2c\mu,
\]
and so by the same argument as above, we have $\mu_{1}^{-}\left(\f\right)\le \mu_2^-(\f) \leq c\mu$, completing the inductive step, and proving the proposition.
\end{proof}

\section{Wrapping up the proof of Proposition \ref{prop:implies Theorem main}}
\label{sec:culmination}

Proposition \ref{prop:implies Theorem main} follows easily by combining
Propositions \ref{prop:mu large r large} and \ref{prop:mu small}
with the corresponding bootstrapping lemmas. We recall the statement
of Proposition \ref{prop:implies Theorem main} for the convenience
of the reader.
\begin{prop*}
There exist absolute constants $c_{1},c_{2}>0$ such that the following
holds. Let $\mu\le\frac{1}{2}$, let $\epsilon\le c_{1}\mu$, and
let $\f\subset\pn$ be a family with $I\left[\f\right]=I\left[\l_{\mu}\right]+\epsilon$
and $\mu\left(\f\right)=\mu$. Then there exists a family $\g$
weakly isomorphic to $\f$ such that one of the following holds.
\begin{itemize}
\item Case (1): $c_{2}\mu_{1}^{-}\left(\g\right)+\frac{1}{2}\epsilon_{1}^{+}\left(\g\right)\le\epsilon$, or
\item Case (2): $c_{2}\mu\left(\g\backslash\s_{\left\{ 1,2\right\} }\right)+\frac{1}{4}\epsilon_{1,2}^{++}\left(\g\right)\le\epsilon$.
\end{itemize}
\end{prop*}
\begin{proof}
Let $\f\subset \pn$ be as in the hypothesis of the proposition. Let $\g$ be a family weakly isomorphic to $\f$, satisfying
\[
\mu_{1}^{-}\left(\g\right)\le\mu_{2}^{-}\left(\g\right)\le\cdots\le\mu_{n}^{-}\left(\g\right)\le\mu_{n}^{+}\left(\g\right)\le\cdots\le\mu_{1}^{+}\left(\g\right).
\]
Lemma \ref{lem:general bootstrapping} implies that either Case (1) or Case (2) holds if
$\mu_{2}^{-}\left(\g\right)\le \tfrac{1}{6}\mu$, provided $c_{2}\leq \tfrac{1}{6}$. By Proposition \ref{prop:mu small},
there exists $d >0$ such that the inequality $\mu_{2}^{-}\left(\g\right)\le \tfrac{1}{6}\mu$
holds provided $\mu\left(\g\right)\le \tfrac{1}{4}+d$, and provided
$c_{1}\le d$. Thus, by Lemma \ref{lem:general bootstrapping}, Case (1) or Case (2) holds for any family $\f$
satisfying $\mu\left(\f\right)\le\tfrac{1}{4}+d$.

We may henceforth assume that $\frac{1}{4}+d\le\mu\left(\f\right)\le\frac{1}{2}$.
By Proposition \ref{prop:mu large r large}, we have $\mu_{1}^{-}\left(\g\right)\le d$
provided $c_{1}$ is sufficiently small depending on $d$. Hence, by
Lemma \ref{lem:bootstrapping r large}, we have $c_{2}\mu_{1}^{-}(\g) +\frac{1}{2}\epsilon_{1}^{+}(\g) \le\epsilon$,
provided $c_{2}\le2$, so Case (1) holds. This completes the proof.
\end{proof}

\begin{remark}
For ease of exposition, we have not attempted to optimize the value of the absolute constant $C$ given by our proof of Theorem \ref{thm:main}. (It can be checked that our proof, as written, yields $C = 2^{6\cdot 2^{360}}$. This can easily be reduced to $C = 2^{360}$, by introducing an extra constant into the statement of Proposition \ref{prop:mu small}.) Unfortunately, it does not seem possible to modify our approach to obtain $C=2$ (see Conjecture \ref{conj:exact} below).
\end{remark}

\section{Conclusion and open problems}
\label{sec:conc}
As mentioned above, we conjecture that Theorem \ref{thm:main} holds with $C=2$.
\begin{conj}
\label{conj:exact}
If $\f\subset \pn$ and $\l \subset \p([n])$ is the initial segment of the lexicographic ordering with $|\l|=|\f|$, then there exists a family $\g \subset \p([n])$ weakly isomorphic to $\l$, such that
$$|\f \Delta \g| \leq 2(|\partial \f| - |\partial \l|).$$
\end{conj}
More generally, it would be of interest to determine more precisely the behaviour of the function
$$s(n,m,l) := \max \{ \min\{|\f \Delta \g|:\ \g \cong \l\}:\ \f \subset \pn,\ |\f| = m,\ |\partial \f| \leq |\partial \l| + l\},$$
where $\mathcal{L}$ denotes the initial segment of the lexicographic ordering of size $m$.

\subsection*{Acknowledgements}
We would like to thank the anonymous referees for their careful reading of the paper, and for their helpful suggestions.

\bibliographystyle{amsplain}

\begin{dajauthors}
\begin{authorinfo}[david]
  David Ellis\\
  School of Mathematical Sciences,\\
  Queen Mary, University of London,\\
  Mile End Road,\\
  London,\\
  E1 4NS,\\
  United Kingdom.\\
  d\imagedot{}ellis\imageat{}qmul\imagedot{}ac\imagedot{}uk
\end{authorinfo}
\begin{authorinfo}[nathan]
  Nathan Keller\\
  Department of Mathematics,
  Bar Ilan University,\\
  Ramat Gan,\\
  5290002,\\
  Israel.\\
  nathan\imagedot{}keller27\imageat{}gmail\imagedot{}com
\end{authorinfo}
\begin{authorinfo}[noam]
  Noam Lifshitz\\
  Department of Mathematics,\\
  Bar Ilan University,\\
  Ramat Gan,\\
  5290002,\\
  Israel.\\
  noamlifshitz\imageat{}gmail\imagedot{}com
\end{authorinfo}
\end{dajauthors}

\end{document}